\documentclass{article}
\usepackage{graphicx} 
\usepackage{ytableau}
\usepackage{tikz}
\definecolor{Green}{HTML}{068a28}
\usepackage[english]{babel}
\usepackage[utf8x]{inputenc}
\usepackage[T1]{fontenc}
\usepackage[rightcaption]{sidecap}
\usepackage{float}
\usepackage[shortlabels]{enumitem}
\usetikzlibrary {arrows.meta}
\usepackage{xparse}
\usepackage{mathtools}
\usepackage{relsize}

\usepackage[a4paper,top=3cm,bottom=2cm,left=3cm,right=3cm,marginparwidth=1.75cm]{geometry}

\usepackage{amsmath,amsthm,amssymb}
\usepackage{cleveref}

\usetikzlibrary{calc, positioning, decorations.markings, bending}

\tikzset{vtx/.style={circle, fill, inner sep=1.5pt}}
\tikzset{pdshift/.style={yscale=-1,shift={(-0.5,-0.5)}}}

\newcommand{\makedecoratedstyle}[2]{
	\tikzset{#1/.style={postaction={decorate}, 
        decoration={markings, mark=at position 0.5 with {#2}}}}
}
\makedecoratedstyle{upline}{
	\draw (0,-0.5) -- (0,0.5);
}
\makedecoratedstyle{flatline}{
        \draw (-0.5,0) -- (0.5,0);
}
\makedecoratedstyle{cross}{
	\draw (0,-0.5) -- (0,0.5);
	\draw (-0.5,0) -- (0.5,0);
}
\makedecoratedstyle{pdcap}{
	\draw (-0.5,0) -- (-0.3,0) to [bend right=45] (0,0.3) -- (0,0.5);
}
\makedecoratedstyle{othercap}{
	\draw (0,-0.5) -- (0,-0.3) to[bend left=45] (0.3,0) -- (0.5,0);
}
\makedecoratedstyle{shadedsquare}{
	\fill[draw=none,fill opacity=0.5] (-0.5,-0.5) --++ (1,0) --++ (0,1) --++ (-1,0) -- cycle;
}

\NewDocumentCommand{\bumps}{ O{} O{#1} m }{
	\foreach \i/\j in {#3} {
  		\draw[pdcap,#1] (\j-0.5,-\i+0.5);
        \draw[othercap,#2] (\j-0.5,-\i+0.5);
	}
}
\NewDocumentCommand{\crosses}{ O{} O{#1} m }{
	\foreach \i/\j in {#3} {
  		\draw[flatline,#1] (\j-0.5,-\i+0.5);
        \draw[upline,#2] (\j-0.5,-\i+0.5);
	}
}
\newcommand{\pdcaps}[2][]{
	\foreach \i/\j in {#2} {
  		\draw[pdcap,#1] (\j-0.5,-\i+0.5);
	}
}
\newcommand{\shadedsquares}[2][]{
	\foreach \i/\j in {#2} {
  		\draw[shadedsquare,#1] (\j-0.5,-\i+0.5);
	}
}

\title{Inversions Tableaux}
\author{Ilani Axelrod-Freed}
\date{July 2024}

\newtheorem{theorem}{Theorem}[section]

\newtheorem{lemma}[theorem]{Lemma}
\newtheorem{corollary}[theorem]{Corollary}
\newtheorem{proposition}[theorem]{Proposition}
\theoremstyle{definition}
\newtheorem{definition}[theorem]{Definition} 
\newtheorem{example}[theorem]{Example}
\newtheorem{question}[theorem]{Question}
\newtheorem{remark}[theorem]{Remark}
\newtheorem{observation}[theorem]{Observation}

\usepackage[backend=bibtex, maxnames=99, minnames=99]{biblatex}

\begin{document}
\newcommand{\Inv}{\mathrm{Inv}}
\newcommand{\IT}{\mathcal{IT}}
\newcommand{\RP}{\mathcal{RP}}
\newcommand{\ID}[1]{\mathcal{ID}_{#1}}

\maketitle

\begin{abstract}
We introduce inversions tableaux, a new combinatorial model for Schubert polynomials and Stanley symmetric functions that directly specializes to semi-standard Young tableaux in the Grassmannian case. 
They are a modification of the balanced staircase tableaux of Edelman and Greene. We explicitly describe inversions tableaux that correspond to the lexicographically minimal and maximal monomials in each Schubert polynomial and characterize the unique inversions tableau for dominant permutations. 
We also characterize the action of generalized chute moves on inversions tableaux, and establish related background that will be used to prove Rubey's chute moves conjecture in upcoming work. 



\end{abstract}

\section{Introduction}

Schubert polynomials play in important role in algebraic geometry and algebraic combinatorics. They were introduced by Lascoux and Schützenberger in 1982 \cite{LascouxS}. Bernstein--Gelfand--Gelfand \cite{bernstein1973schubert} and Demazure \cite{demazure1974} showed that Schubert polynomials represent cohomology classes of the flag manifold $GL(n)/B$, where $B$ is the Borel subgroup of upper triangular matrices. In particular, Schubert polynomials in $n$ variables are indexed by permutations in the symmetric group $S_n$ and form a basis for the coinvariant algebra $\mathbb{C}[x_1, \dots, x_n]/I$, where $I$ is the ideal generated by elementary symmetric polynomials. Schubert polynomials are also frequently modeled combinatorially in terms of compatible sequences \cite{billey1993some, FOMIN1994196}, pipe dreams (also called RC-graphs) \cite{bergeron1993rc, fomin1996yang}, bumpless pipe dreams \cite{lam2021back} and balanced labellings \cite{balancedlabellings}. In \cite{billey1993some} Bergeron and Billey define an operation called chute moves on pipe dreams and use it to prove the equivalence of the combinatorial and algebraic definitions of Schubert polynomials. Building off this, Rubey definited a poset on pipe dreams with covering relations given by generalized chute moves, and conjectured this poset is a lattice \cite{RUBEY}.

Given two Schubert polynomials $\mathfrak{S}_u$ and $\mathfrak{S}_v$, we can express their product as a linear combination of other Schubert polynomials
    $$\mathfrak{S}_u \mathfrak{S}_v = \sum\limits_{w \in S_n} c^w_{u,v} \mathfrak{S}_w.$$
These $c^w_{u,v}$ coefficients are the Schubert structure coefficients, also called the Schubert Littlewood--Richardson coefficients. One of the main open questions in Schubert calculus is to find a combinatorial interpretation for these coefficients.

For $w$ in a special class of permutations called Grassmannian permutations, the Schubert polynomial $\mathfrak{S}_w$ becomes a Schur polynomial. 
Schur polynomials have been studied extensively (see \cite{EC2}  for a good introduction) and play an essential role in the theory of symmetric functions, representations of the symmetric group, geometry of the Grassmannian, and more. Many known properties of Schubert polynomials come from generalizing properties known first for Schur polynomials. Schur polynomials have a combinatorial interpretation in terms of semistandard Young tableaux (SSYT). There is a well understood Littlewood--Richardson rule which gives the Schur structure coefficients in terms of the numbers of a certain subset of skew SSYT \cite{littlewood1934group, stembridge2002concise}.

In this paper we give a combinatorial interpretation for Schubert polynomials using certain tableaux. In \cite{EG}, Edelman and Greene give a bijection between chains in weak Bruhat order and balanced staircase tableaux. In this paper, we consider staircase shape tableaux with numbers in only a specific subset of boxes. We define such a tableaux to be an \emph{inversions tableau} if these numbers satisfy a new generalization we provide of the balanced condition from \cite{EG}. Inversions tableaux generalize the SSYT formula from the Schur case to the Schubert case and allow us to see several nice properties of Schubert polynomials. Inversions tableaux are defined in a similar way to the balanced labellings of Fomin--Greene--Reiner--Shimozono \cite{balancedlabellings}. However they additionally allow us to directly compare permutations in weak Bruhat order, and to use more local properties to check when a tableau is balanced. Through our main bijection with pipe dreams in Section \ref{section:inversions tableaux}, we also obtain a bijection for pipe dreams and balanced labellings. Since the original release of this work, Sara Billey informed us that an equivalent construction called flag inversion fillings were studied in Elizabeth Kelly's (unpublished) PhD thesis, which has since been made available online \cite{Kelly.thesis}.

From inversions tableaux, we define a new class of tableaux of the same shape which they are in bijection with called Lehmer tableaux. We can directly compare two pipe dreams in the generalized chute moves poset by comparing their Lehmer tableaux entry-wise. In this paper, we characterize how chute moves act on inversions and Lehmer tableaux, and show that if two pipe dreams are related via generalized chute moves, then their Lehmer tableaux are comparable entry-wise. In subsequent work with Defant, Mularczyk, Nguyen, and Tung, \cite{proofofrubey} we show the converse of the statement holds, and use these properties of Lehmer tableaux to prove Rubey's lattice conjecture from \cite{RUBEY}.


In Section \ref{section:background} we provide background on Schur and Schubert polynomials and their standard combinatorial interpretations.
In Section \ref{section:inversions tableaux} we define inversions tableaux and show they give Schubert polynomials by proving a bijection with pipe dreams. We also extend the construction of inversions tableaux to a combinatorial model for Stanley symmetric functions.
In Section \ref{section:lehmer code} we show how inversions tableaux give a particularly nice way to read the Lehmer code of a permutation, and from there give the formula and corresponding inversions tableaux for the maximal and minimal monomials of each Schubert polynomial. In doing so we also fully characterize the inversions tableaux for dominant permutations. 
In Section \ref{section:Grassmannian} we show that inversions tableaux for Grassmannian permutations are exactly reverse SSYT. Further for inverse Grassmannian permutations we show their inversions tableaux can be read easily as certain flag SSYT.
Finally in Section~\ref{section:chute moves}, we show how generalized chute moves act on inversions tableaux, relate this to a new ordering we define on permutations based on their Lehmer codes, and introduce Lehmer tableaux as a way of comparing two pipe dreams in Rubey's chute move poset.

\section{Background}
\label{section:background}
We will begin by setting the conventions for the symmetric group and weak Bruhat order that we will use in this paper, as well as defining Schur and Schubert polynomials and the combinatorial objects associated with them.

Let $S_n$ be the symmetric group on $n$ elements. Given a permutation $w\in S_n$ written in one line notation, the \textit{adjacent transposition} $s_i$ acts on the left and swaps the positions of numbers $i$ and $i+1$. Note that $\{s_i\}_{i=1}^{n-1}$ generates $S_n$ as a group. We say that $(i,j)$ is an \emph{inversion} if $i<j$ but $w_j < w_i$. Let $\Inv(w)$ be the set of inversions of $w$. A valid inversion set uniquely determines a permutation. The \emph{length} of $w$ is $l(w) = \# \Inv(w)$. A \emph{reduced word} for $w$ is a sequence of indices $a_1, a_2, \dots a_{l(w)}$ such that $w = s_{a_{l(w)}} \cdots s_{a_2} s_{a_1}$.

This is also the number of generators in a reduced word for $w$. We write $\text{Id}$ for the identity and let $w_0 = n (n-1) \cdots 2 1$ denote the unique permutation of maximal length in $S_n$.

\begin{definition}
    The (left) weak Bruhat order on $S_n$ is a partial ordering of the elements given by the following covering relation: For $u, w \in S_n$, we have $u \lessdot w$ if $w = s_i u$ for some $i$, and $l(w')=l(w)+1$.
\end{definition}
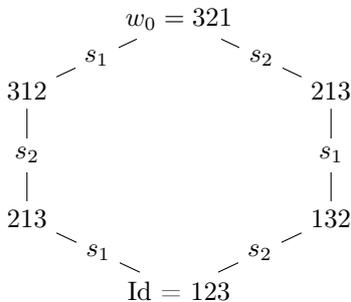
\begin{figure}[h]
    \centering
    \begin{tikzpicture}
   \node[fill=white] (a) at (0,-0.05) {\text{Id} = 123};
    \node[fill=white] (b) at (2,.9)  {132};
    \node[fill=white] (c) at (2,2.6) {213};
    \node[fill=white] (d) at (0,3.55) {$w_0=321$};
    \node[fill=white] (e) at (-2,2.6) {312};
    \node[fill=white] (f) at (-2,.9) {213};
\draw (a) -- (f) node [midway, fill=white] {$s_1$};
\draw (a) -- (b) node [midway, fill=white] {$s_2$};
\draw (f) -- (e) node [midway, fill=white] {$s_2$};
\draw (b) -- (c) node [midway, fill=white] {$s_1$};
\draw (c) -- (d) node [midway, fill=white] {$s_2$};
\draw (e) -- (d) node [midway, fill=white] {$s_1$};
    \end{tikzpicture}
    \caption{The (left) weak Bruhat order on $S_3$.}
    \label{fig:weak-bruhat}
\end{figure}

For permutations $u,w \in S_n$, we have that $u<w$ in weak Bruhat order exactly when $\Inv(u) \subset \Inv(w)$.

Next we define Schubert polynomials via divided difference operators. The $i^{th}$ \emph{divided difference operator} $\partial_i:\mathbb{Z}[x_1,x_2,\cdots]\to \mathbb{Z}[x_1,x_2,\cdots]$ is given by
\[\partial_i(f) = {f-s_i(f)\over x_i - x_{i+1}}\]
where $s_i(f)$ is the polynomial obtained from $f$ by swapping $x_i\leftrightarrow x_{i+1}$. Schubert polynomials are indexed by permutations and are uniquely determined by the following:
\begin{align*}
\mathfrak{S}_{w_0} &= x_1^{n-1}x_2^{n-2}\cdots x_{n-1};\\	
\partial_i\mathfrak{S}_{w} & =  \begin{cases}
	\mathfrak{S}_{w s_i}&\text{if } l(w s_i) < l(w)\\
	0&\text{otherwise.}
\end{cases}\end{align*}

Schubert polynomials can also be defined combinatorially using pipe dreams.

\begin{definition}
   For $w\in S_n$, index squares in an $n \times n$ grid using matrix coordinates. Tile all squares in $\{(i,j)|i+j = n +1 \}$ with
   \begin{tikzpicture}
	\draw (-0.2,0) -- (-0.1,0) to [bend right=45] (0,0.1) -- (0,0.2);
\end{tikzpicture}
   , and tile all squares in $\{(i,j)|i+j < n +1 \}$ with either a 
\begin{tikzpicture}
    \draw (0,-0.2) -- (0,-0.1) to[bend left=45] (0.1,0) -- (0.2,0);
	\draw (-0.2,0) -- (-0.1,0) to [bend right=45] (0,0.1) -- (0,0.2);
\end{tikzpicture}
or
\begin{tikzpicture}
    \draw (0,-0.2) -- (0,0.2);
	\draw (-0.2,0) -- (0.2,0);
\end{tikzpicture}
, so that the resulting pipe starting at spot $i$ on the left (the $i^{th}$ pipe) ends at spot $w_i$ on the top for each $1 \leq i \leq n$. This is a \emph{pipe dream} for $w$. 
The \emph{weight} of the pipe dream $P$ is $wt(P)=(d_1,d_2, \dots , d_n)$ where $d_i$ is the number of crossing tiles in row $i$.  
The pipe dream is \emph{reduced} if no two pipes cross each other more than once. Let $\mathcal{RP}(w)$ denote the set of reduced pipe dreams for $w$. 
\end{definition}

Bergeron and Billey proved that each Schubert polynomial can be expressed as a generating function over pipe dreams.

\begin{theorem}[\cite{billey1993some}]
 For a permutation $w \in S_n$, the \emph{Schubert polynomial} for $w$ is
$$\mathfrak{S}_w = \sum_{P \in \RP(w) \ }  x^{wt(P)}$$ 
where for $wt(P) = (d_1, \dots , d_n)$, we define $x^{wt(P)} = x_1^{d_1} x_2^{d_2} \cdots x_n^{d_n}$.
\end{theorem}

\begin{example}
\label{example:pipe dreams}
$$
\begin{matrix}
\mathfrak{S}_{431562} 
   =& x_1^3 x_2^2 x_4 x_5 &+& x_1^3 x_2^2 x_3 x_5 &+& x_1^3 x_2^2 x_3 x_4 \\
  & & & & & \\
   & \scalebox{.5}{ 
  \begin{tikzpicture}
\crosses{1/1,1/2,1/3,2/1,2/2,4/1,5/1}
\bumps{1/4,1/5,2/3,2/4,3/1,3/3,3/2,4/2}
\pdcaps{1/6,2/5,3/4,4/3,5/2,6/1}

\foreach \i in {1,...,6} {
	\node[left] at (0,-\i+0.5) {\Large \i};
    \node[above] at (\i-0.5,0) {\Large \i};
}
\end{tikzpicture}
}
  &&\scalebox{0.5}{ 
  \begin{tikzpicture}
\crosses{1/1,1/2,1/3,2/1,2/2,3/2,5/1}
\bumps{1/4,1/5,2/3,2/4,3/1,3/3,4/1,4/2}
\pdcaps{1/6,2/5,3/4,4/3,5/2,6/1}

\foreach \i in {1,...,6} {
	\node[left] at (0,-\i+0.5) {\Large \i};
    \node[above] at (\i-0.5,0) {\Large \i};
}
\end{tikzpicture}
}
 && \scalebox{0.5}{ 
  \begin{tikzpicture}
\crosses{1/1,1/2,1/3,2/1,2/2,3/2,4/2}
\bumps{1/4,1/5,2/3,2/4,3/1,3/3,4/1,5/1}
\pdcaps{1/6,2/5,3/4,4/3,5/2,6/1}

\foreach \i in {1,...,6} {
	\node[left] at (0,-\i+0.5) {\Large \i};
    \node[above] at (\i-0.5,0) {\Large \i};
}
\end{tikzpicture}
}
\end{matrix}
$$
\end{example}

Schubert polynomials are a generalization of Schur polynomials, which are indexed by Young diagrams. Like permutations, Young diagrams also come with a partial ordering: given Young diagrams $\lambda$ and $\mu$, we say $\lambda < \mu$ if $\lambda \subset \mu$. This produces a poset called \emph{Young's lattice}. Chains starting at $\emptyset$ in Young's Lattice correspond to standard Young tableaux (SYT), Young diagrams with a unique number in each box that have values increasing along rows and columns. If the chain starts at a tableaux $\mu \neq \emptyset$ and goes to $\lambda$, then we get a SYT of \emph{skew shape} $\lambda/\mu$.

\begin{example}
A SYT of shape $\lambda=(4,3,3)$, and another of skew shape $\lambda/\mu =(4,4,3)/(3,1)$:
    \begin{eqnarray*}
\ytableausetup{baseline}
\begin{ytableau}
*(white) 1 & *(white)  3 & *(white) 4 & *(white) 10\\
*(white)  2 & *(white) 5 & *(white) 8 \\
*(white) 6 & *(white) 7 & *(white) 9 \\
\end{ytableau}
& \ \ \ \ &
\begin{ytableau}
\none & \none & \none & *(white) 6\\
\none & *(white) 2 & *(white) 4 & *(white) 7 \\
*(white) 1 & *(white) 3 & *(white) 5 \\
\end{ytableau}
\end{eqnarray*}
\end{example}

If instead of putting an $i$ in the box we add at the $i^{th}$ step of our chain, we add numbers in weakly increasing fashion, following the right restrictions, we can now obtain semi-standard Young tableaux.

\begin{definition}
A semi-standard Young tableau (SSYT) of shape $\lambda$ is a Young diagram of shape $\lambda$ with numbers in each box which are weakly increasing along rows and strictly increasing down columns. The \emph{weight} of a SSYT $T$ is $wt(T) := (m_1,m_2,m_3,\ldots)$, where $m_i$ is the number of boxes in $T$ filed with the number $i$. The \emph{Schur polynomial} indexed by $\lambda$ is defined in terms of SSYT as
$$s_{\lambda}(x_1, \dots ,x_n) := \sum_{T \in SSYT} x^{wt(T)},$$
where the sum is over SSYT whose boxes have entries in $\{1, \dots, n\}$.
\end{definition}

We can also take the same definition allowing for skew shapes $\lambda/\mu$, in which case we get the \emph{skew Schur polyomial} $s_{\lambda/\mu}(x_1, \dots, x_n)$.
For a box at location $(i,j)$ in the tableau $T$ (using matrix coordinates), we denote its entry as $T(i,j)$.

\begin{example}
An SSYT of shape $\lambda = (5,3,2)$ and weight $(1, 4, 3, 2)$ which contributes the monomial $x_1 x_2^4 x_3^3 x_4 ^2$ to the sum for $s_{\lambda}(x_1, x_2, x_3, x_4)$: 
\begin{eqnarray*}
\ytableausetup{baseline}
\begin{ytableau}
*(white) 1 & *(white)  2 & *(white) 2 & *(white) 2 & *(white)  4\\
*(white)  2 & *(white) 3 & *(white) 3 \\
*(white) 3 & *(white) 4 \\
\end{ytableau}
\end{eqnarray*}
\end{example}


\section{Inversions Tableaux}
\label{section:inversions tableaux}

In this section, we define inversions tableaux and provide a bijection with pipe dreams. First we will define inversions diagrams which give us a way to frame weak Bruhat order via containment of diagrams, similar to Young's lattice. We will show how to represent maximal chains in weak Bruhat order as balanced staircase tableaux, and how this may be softened to what we call a ``weakly balanced'' condition akin to the loosening from SYT to SSYT. Finally, we define inversions tableaux as weakly balanced tableaux satisfying certain additional properties and provide a bijection with pipe dreams.

\begin{definition}
    Given a permutation $w \in S_n$, the \emph{inversions diagram} for $w$ (or of shape $w$), $T_w$, is the $(n-1)\times(n-1)$ staircase tableaux where we shade in the box indexed $(i,j)$ exactly when $(i,j) \in \Inv(w)$.
\end{definition}
Throughout the paper, we will orient the staircase tableaux as in the example below, with row labels $1, \dots, n-1$ increasing bottom to top and columns labels $2, \dots , n$ increasing left to right, and $(i,j)$ denoting the entry in row $i$ and column $j$. The staircase then contains all boxes indexed by $\{(i,j) | 1 \leq i < j \leq n\}$.

\begin{example}
\label{fig:inversions diagram}
The inversions diagram for $w = 4\textcolor{green}{\underline{\textcolor{black}{5}}}162\textcolor{green}{\underline{\textcolor{black}{3}}}$ is 
\begin{eqnarray*}
\ytableausetup{baseline}
\begin{ytableau}
  \none & \none & \none & \none & *(white) & \none[5] \\
  \none & \none & \none & *(green) & *(green) & \none[4] \\
  \none & \none & *(white) & *(white) & *(white) & \none[3] \\
  \none & *(green) & *(white) & *(green) & *(green) & \none[2] \\
  *(white) & *(green) & *(white) & *(green) & *(green) & \none[1] \\
  \none[2] & \none[3] & \none[4] & \none[5] & \none[6] \\
  \end{ytableau}
\end{eqnarray*}

Spots 2 and 6 of $w$ (underlined in green) contain numbers 5 and 3 respectively. These are out of order ($5<3$), so box $(2,6)$ is shaded.

\end{example}

\begin{remark}
    Given an inversions tableau for a permutation $w$, we can read off the permutation by the fact that $w_i = i + \#(\text{shaded boxes in row }i) - \#(\text{shaded boxes in column }i)$.
\end{remark}

Let $\ID{n}$ be the set of inversions diagrams of permutations in $S_n$.
Since shading boxes is equivalent to recording the set of inversions, we can again reconstruct the permutation given its inversions diagram. Similarly, permutations $u,w \in S_n$ satisfy $u<w$ in weak Bruhat order exactly when the shaded boxes of $T_u$ are a subset of those in $T_w$. Thus $\ID{n}$ with the poset structure given by inclusion of diagrams is equivalent to weak Bruhat order. 

In \cite{EG} Edelman and Greene show a correspondence between maximal chains in weak Bruhat order and balanced tableaux.

\begin{definition}
    In an inversions tableau, a \emph{hook} consists of a box $(i,j)$ together with all boxes to the left of $(i,j)$ in the same row and all boxes above $(i,j)$ in the same column. Formally:
    \begin{align*}
        \text{hook}_{i,j} &= \{ (i,j)\} \cup \{(i',j) \mid i' > i\} \cup \{(i,j') \mid j' < j\}
    \end{align*}
\end{definition}

\begin{definition}
   A staircase tableau with distinct numbers in each box is \emph{balanced} if for each hook, the value in the corner of the hook is the median of all the values in the hook. 
\end{definition}
For staircase tableaux indexed as in Example \ref{fig:inversions diagram}, we denote the entry in box $(i,j)$ as $T(i,j)$. 

\begin{example} A balanced (left) and unbalanced (right) hook, with the corner of each hook shaded.
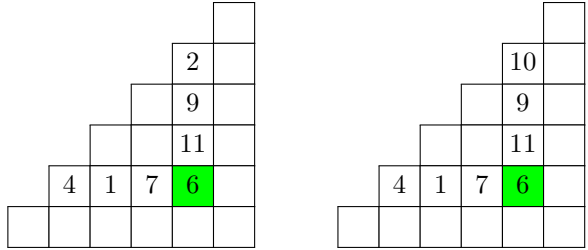
\phantom{a}
\begin{eqnarray*}
\ytableausetup{baseline}
\begin{ytableau}
\none & \none  & \none  & \none & \none & *(white)\\
\none & \none  & \none  & \none & *(white) 2 & *(white) \\
\none & \none  & \none  & *(white) & *(white) 9 & *(white)\\
\none & \none  & *(white)  & *(white) & *(white) 11 & *(white) \\
\none & *(white) 4 & *(white) 1 & *(white) 7 & *(green) 6 & *(white) \\
*(white) & *(white)  & *(white)  & *(white) & *(white) & *(white) & \none & \none \\
\end{ytableau}
\ytableausetup{baseline}
\begin{ytableau}
\none & \none  & \none  & \none & \none & *(white)\\
\none & \none  & \none  & \none & *(white) 10 & *(white) \\
\none & \none  & \none  & *(white) & *(white) 9 & *(white)\\
\none & \none  & *(white)  & *(white) & *(white) 11 & *(white) \\
\none & *(white) 4 & *(white) 1 & *(white) 7 & *(green) 6 & *(white) \\
*(white) & *(white)  & *(white)  & *(white) & *(white) & *(white) \\
\end{ytableau}
\end{eqnarray*}
\end{example}

 We reframe the bijection from \cite{EG} in terms of inversions diagrams:

\begin{proposition}
\label{prop:balanced chain bij}
\cite[Theorem 4.2]{EG}
There is a bijection between maximal chains in $\ID{n}$ and balanced staircase tableaux of size $n-1$; namely a maximal chain in $\ID{n}$ corresponds to the tableau $T$ where $T(i,j)$ is the number of the step where box $(i,j)$ was added in the chain (i.e. the step the transposition $(i,j)$ was added in weak Bruhat order). 
\end{proposition}

To prove Proposition \ref{prop:balanced chain bij}, Edelman and Greene use the following two lemmas, which we again restate for convenient use in the inversions tableaux context. The color coding for this section matches with Figure \ref{fig:balanced condition}.

\begin{lemma}\label{lemma:balanced condition} \cite[Lemma 4.3]{EG}
 (The Strict Rectangle Rule)  A staircase tableaux is balanced if and only if for all $i<j<k$, the number $\textcolor{orange}{T(i,k)}$ is between $\textcolor{green}{T(i,j)}$ and $\textcolor{green}{T(j,k)}$.
\end{lemma}

\begin{figure}[h]
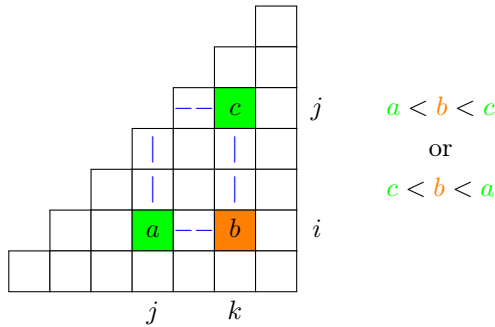

\centering
\begin{ytableau}
\none & \none & \none  & \none  & \none & \none  & *(white) & \none \\
\none & \none & \none  & \none  & \none & *(white)  & *(white) & \none \\
\none & \none & \none  & \none  & *(white) \textcolor{blue}{--} & *(green) c  & *(white) & \none[j] 
& \none & \none & \none[\textcolor{green}{a} < \textcolor{orange}{b} < \textcolor{green}{c}]
\\
\none & \none & \none  & *(white) \textcolor{blue}{|} & *(white)  & *(white) \textcolor{blue}{|} & *(white) & \none 
& \none & \none & \none[\text{or}]
\\
\none & \none & *(white)  & *(white) \textcolor{blue}{|} & *(white)  & *(white) \textcolor{blue}{|} & *(white)& \none 
& \none & \none & \none[\textcolor{green}{c} < \textcolor{orange}{b} < \textcolor{green}{a}]\\
\none & *(white) & *(white)  & *(green) a & *(white) \textcolor{blue}{--} & *(orange) b & *(white) & \none[i]\\
*(white) & *(white) & *(white)  & *(white)  & *(white) & *(white) & *(white) & \none \\
\none & \none & \none  & \none[j]  & \none & \none[k]\\
\end{ytableau}
 \caption{Illustration of the (Strict) Rectangle Rule: The base box \textcolor{orange}{(orange)} has entry $\textcolor{orange}{b=T(i,k)}$, and the opposite corners \textcolor{green}{(green)} have entries $\textcolor{green}{a=T(i,j)}$ and $\textcolor{green}{c=T(k,j)}$.}
    \label{fig:balanced condition}
\end{figure}

 We will sometimes refer to Lemma \ref{lemma:balanced condition} as the ``(Strict) Rectangle Rule'' due to the visual interpretation illustrated in Figure \ref{fig:balanced condition}: The squares involved in Lemma \ref{lemma:balanced condition} fall in three corners of a rectangle, with the ``base box'', box $\textcolor{orange}{(i,k)}$, on the bottom right, the ``opposite corners'', boxes $\textcolor{green}{(i,j)}$ and $\textcolor{green}{(j,k)}$, on the bottom left and top right respectively, and with the fourth corner of the rectangle falling just outside of the main diagonal. We can further relate this condition to rules on the inversion set of a permutation: 

\begin{corollary}
    \label{cor:inversion shape}
An inversions diagram is valid if and only if for any rectangle as in Figure \ref{fig:balanced condition}, if both the opposite corners are shaded (resp. not shaded) then the base box must be as well.
In other words, set $I \subset \{(i,j)| 1 \leq i < j \leq n\}$ is a valid inversion set for some permutation if and only if for $i<j<k$, we have $\textcolor{green}{(i,j)}$ and $\textcolor{green}{(j,k)}$ both in $I$ implies $\textcolor{orange}{(i,k)} \in I$ and both not in $I$ implies $\textcolor{orange}{(i,k)} \notin I$.
\end{corollary}

\begin{proof}
 For a permutation $w$, take any chain through $T_w$ in $\ID{n}$. Then $(i,j) \in \Inv(w)$ if and only if $T(i,j) \leq l(w)$. Thus inclusion corresponds exactly to the conditions above.
\end{proof}


We now introduce the concept of weakly balanced tableaux. This is similar
 to the balanced condition for hooks in balanced labellings from \cite{balancedlabellings}, but here with hooks including all boxes in the staircase, not just those with (non-zero) numbers inside.

\begin{definition}
    A staircase tableau $T$ with not necessarily distinct entries in each box is \emph{weakly balanced} if for all $i<j$, the entry $T(i,j)$ is the median value (counting multiplicities) of all the entries in the hook $\text{hook}_{i,j}$.
\end{definition}

Much like regular balanced tableaux, we can identify if a tableaux is weakly balanced if it follows a similar rectangle rule.

\begin{lemma}
\label{lemma:balanced}
    Let $T$ be a staircase tableau with not necesessarily distinct entries in each box. Then the following are equivalent.
    \begin{enumerate}
        \item $T$ is weakly balanced.
        \item $T$ satisfies the \emph{Rectangle Rule:} 
        
        For all $i<j<k$, the number $\textcolor{orange}{T(i,k)}$ is weakly between $\textcolor{green}{T(i,j)}$ and $\textcolor{green}{T(j,k)}$, i.e. in Figure \ref{fig:balanced condition} we have $\textcolor{green}{a} \leq \textcolor{orange}{b} \leq \textcolor{green}{c}$ 
   or $\textcolor{green}{c} \leq \textcolor{orange}{b} \leq \textcolor{green}{a}$.
    \item There exists a total ordering $\lessdot$ on all the boxes in the staircase such that 
\begin{itemize}
    \item If $u \lessdot v$ then $T(u) \leq T(v)$.
    \item Labeling the boxes according to this total order produces a balanced tableaux. 
\end{itemize}
    \end{enumerate}
\end{lemma}

For the third condition, we are essentially assigning ``tie breaks'' to entries of the same value which make the tableau balanced.

\begin{proof}
  The proof that 1 and 2 are equivalent is the same as the proof of Lemma 4.3 from \cite{EG}, just replacing some strict inequalities with weak inequalities. 
  
   Now we show that conditions 2 and 3 are equivalent. If $T$ satisfies condition 3, then impose a total ordering on the numbers which makes it balanced. All rectangles for $i<j<k$ in $T$ have corner entries satisfying the strict inequality relations. Returning to our weak order where a number is allowed to appear multiple times, we still maintain loose inequality.

   In the other direction, suppose we have a filling $T$ satisfying condition 2. For each $i$, let $T_i$ be the inversions tableaux with boxes of squares containing numbers $1, \dots, i$ shaded. The weak inequality condition from 2 guarantees that each $T_i$ is a valid inversions diagram. Further, containment gives us that $T_{\text{Id}} < T_{w_1} < T_{w_2} < \dots < T_{w_0}$ is a chain in the poset $\ID{n}$. Thus there is at least one saturated chain in $\ID{n}$ going through all these diagrams, which by Proposition \ref{prop:balanced chain bij} corresponds to a balanced tableaux $T'$. The box labels of $T'$ provide a total ordering for  $T$, and we can recover $T$ by replacing all entries $ht(w_{i-1})+1, \dots ht(w_i)$ of $T'$ with the number $i$.

\end{proof}


Finally, we are prepared to define the main object of our paper -- inversions tableaux.

\begin{definition}\label{Def:??-tableaux}
An \emph{inversions tableau} of shape $w$ is an inversions diagram $T$ of shape $w$ where we put a number $1, \dots , n$ in each shaded box (and all unshaded boxes are assumed to be 0), such that the diagram meets the following rules.

\begin{enumerate}
    \item[(IT1)] (Rectangle Rule:) The tableaux is weakly balanced. 
    \item[(IT2)] No two shaded boxes in a column have the same entry.
    \item[(IT3)] $T(i,i+1) \leq i$ for all $1 \leq i < n$.
\end{enumerate}
\end{definition}

We denote the set of inversions tableaux for the permutation $w \in S_n$ as $\IT(w)$. These give a generating function for Schubert polynomials which is very reminiscent of the SSYT generating function for Schur polynomials. Similarly to SSYT, define the \emph{weight} of an inversions tableaux $T$ to be the tuple $wt(T) := (m_1,\ldots,m_{n-1})$, where $m_i$ is the number of boxes in $T$ with entry $i$. Note we do not count any of the zeros.

Inversions tableaux are equivalent to flagged inversion fillings from \cite{Kelly.thesis}, where it was also shown they satisfy Theorem \ref{thrm:schub formula}. However this paper has the first appearance of (IT3) as a characterization. Inversion fillings are said to be \emph{flagged} if they satisfy the rule (IT3') stated further down.

\begin{theorem}
\label{thrm:schub formula}
The Schubert polynomial for $w \in S_n$ is given by
    $$\mathfrak{S}_w = \sum\limits_{T \in \IT(w)} x^{wt(T)}.$$
\end{theorem}

\begin{example}
\label{example:inversionstableaux}
The inversions tableaux corresponding to the same terms as in Example \ref{example:pipe dreams}.
$$
\begin{matrix}
\mathfrak{S}_{431562} 
   =& x_1^3 x_2^2 x_4 x_5 \ \ \ \ &+& x_1^3 x_2^2 x_3 x_5 \ \ \ \ &+& x_1^3 x_2^2 x_3 x_4 \ \ \ \  \\
  & \begin{ytableau}
  \none & \none & \none & \none & *(green) 5 & \none[5] \\
  \none & \none & \none & *(white) & *(green) 4 & \none[4] \\
  \none & \none & *(white) & *(white) & *(white) & \none[3] \\
  \none & *(green) 2 & *(white) & *(white) & *(green) 2& \none[2] \\
  *(green) 1 & *(green) 1 & *(white) & *(white) & *(green) 1 & \none[1] \\
  \none[2] & \none[3] & \none[4] & \none[5] & \none[6] \\
  \end{ytableau}
 && \begin{ytableau}
  \none & \none & \none & \none & *(green) 5 & \none[5] \\
  \none & \none & \none & *(white) & *(green) 3 & \none[4] \\
  \none & \none & *(white) & *(white) & *(white) & \none[3] \\
  \none & *(green) 2 & *(white) & *(white) & *(green) 2& \none[2] \\
  *(green) 1 & *(green) 1 & *(white) & *(white) & *(green) 1 & \none[1] \\
  \none[2] & \none[3] & \none[4] & \none[5] & \none[6] \\
  \end{ytableau}
 && \begin{ytableau}
  \none & \none & \none & \none & *(green) 4 & \none[5] \\
  \none & \none & \none & *(white) & *(green) 3 & \none[4] \\
  \none & \none & *(white) & *(white) & *(white) & \none[3] \\
  \none & *(green) 2 & *(white) & *(white) & *(green) 2& \none[2] \\
  *(green) 1 & *(green) 1 & *(white) & *(white) & *(green) 1 & \none[1] \\
  \none[2] & \none[3] & \none[4] & \none[5] & \none[6] \\
  \end{ytableau}
\end{matrix}
$$
\end{example}

To prove Theorem \ref{thrm:schub formula}, we will give a bijection between $\IT(w)$ and $\RP(w)$ in which each occurrence of the number $i$ in a box of our inversions tableau corresponds to a crossing in the $i$'th row of our pipe dream. We will do this using an alternative (equivalent) definition of inversions tableaux provided by the next lemma. 

\begin{lemma}\label{lemma:weaker def}
An inversions diagram for $w$ with numbers $1, \dots , n$ in each shaded box is an inversions tableaux of shape $w$ if and only if it satisfies conditions (IT1) and (IT2) from Definition~\ref{Def:??-tableaux}, as well as the following:
\begin{enumerate}
    \item [(IT3')] All numbers in row $i$ are less than or equal to $i$.
\end{enumerate}
\end{lemma}

\begin{proof}
    Condition (IT3) is a stronger condition than (IT3'), so we have immediately that the new conditions of Lemma \ref{lemma:weaker def} imply Definition \ref{Def:??-tableaux}. It remains to show that conditions (IT1), (IT2) and (IT3) together imply (IT3'). We will induct downwards on the row number. For row $n-1$, the only box is on the main diagonal, so, by condition (IT3'), its entry is at most $n-1$. 
    
    We now show for row $k$. We know the box $(k,k+1)$ must have entry $a \leq k$ by condition (IT3'). Given any other box on the $k$'th row, let its entry be $b$ and the entry of the box immediately above it (in row $k+1$) be $c$. Then, by Lemma \ref{lemma:balanced condition}, to satisfy condition (IT1), either $c \leq b \leq a \leq k$, or $a \leq b \leq c$. By the inductive assumption, $c$ is at most $k+1$, so $b$ must be as well. We rule out the case where $b=c=k+1$ because this would violate condition (IT2).
 \end{proof}

Now we can prove Theorem \ref{thrm:schub formula} by providing a weight preserving bijection between pipe dreams and inversions tableaux satisfying rules (IT1),(IT2), and (IT3'). Here, by the weight of a pipe dream, we mean the tuple recording the number of crossings in each row. 

\begin{proof}[Proof of Theorem \ref{thrm:schub formula}]
Given a pipe dream $P \in \RP(w)$, pipes $i$ and $j$ cross exactly if $(i,j) \in \Inv(w)$. Let $r_{ij}$ be the row number in which they cross. Define $\phi(P)$ to be the inversions tableaux $T$ such that $T(i,j) = r_{ij}$ for all $(i,j) \in \Inv(w)$. Then $\phi: \RP(w) \rightarrow \IT(w)$ preserves weights, and we will show it is a bijection.
    First, given any pipe dream $P$, we show that $\phi(P)$ satisfies rules (IT1), (IT2) and (IT3'). 
    
    (IT1): Notice that if we add crossings of our pipe dream $P$ from bottom to top and within each row from left to right, then we get permutations which form a saturated chain from $\emptyset$ to $w$ in weak Bruhat order.  Equivalently, we have a saturated chain in $\ID{}$ from the identity to $T_w$.
    At each step the crossing we add corresponds to adding precisely the box indexed by the two pipe numbers we crossed. We can then extend this chain to a saturated chain from $\emptyset$ to $T_{w_0}$. By Proposition \ref{prop:balanced chain bij}, this gives a balanced tableaux with boxes numbered in the order they were added in the chain. If we reverse the order of all numbers (i.e replace $i$ with $(n-1)n/2-i+1$) then we still have a balanced tableaux. Observe that this is now a total ordering on the entries of $\phi(P)$. In particular, entries of boxes added after $T_w$ in our chain are all 0, and labels from the crossings now correspond to the row the crossing occurred in.

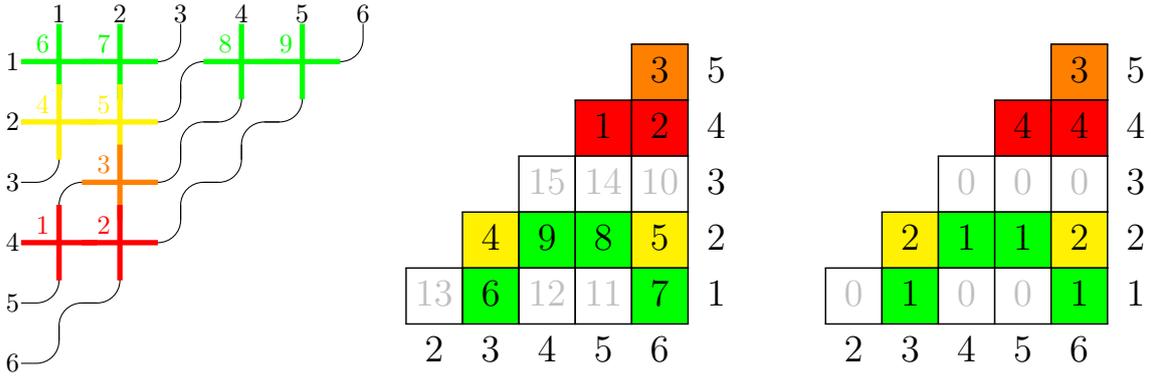
\begin{figure}[h]
\centering
    
$
\begin{matrix}
\begin{tikzpicture}[scale=0.8]
\bumps{1/3,2/3,2/4, 3/1,3/3,5/1}
\pdcaps{1/6,2/5,3/4,4/3,5/2, 6/1}
\crosses[green, line width=2pt]{1/1,1/2, 1/4,1/5}
\crosses[yellow, , line width=2pt]{2/1,2/2}
\crosses[orange, , line width=2pt]{3/2}
\crosses[red, , line width=2pt]{4/1,4/2}
\begin{scope}[pdshift]	
	\node[above left] at (1,1) {$\textcolor{green}{6}$};
    \node[above left] at (2,1) {$\textcolor{green}{7}$};
	\node[above left] at (4,1) {$\textcolor{green}{8}$};
    \node[above left] at (5,1) {$\textcolor{green}{9}$};
    \node[above left] at (1,2) {$\textcolor{yellow}{4}$};
    \node[above left] at (2,2) {$\textcolor{yellow}{5}$};
    \node[above left] at (2,3) {$\textcolor{orange}{3}$};
    \node[above left] at (1,4) {$\textcolor{red}{1}$};
    \node[above left] at (2,4) {$\textcolor{red}{2}$};
\end{scope}
\foreach \i in {1,...,6} {
	\node[left] at (0,-\i+0.5) {\i};
    \node[above] at (\i-0.5,0) {\i};
}
\end{tikzpicture}
&
\ytableausetup{aligntableaux=bottom}
{\Large
\begin{ytableau}
   \none & \none & \none & \none & *(orange) 3 & \none[5] \\
  \none & \none & \none & *(red) 1 & *(red) 2 & \none[4] \\
  \none & \none & *(white) \textcolor{lightgray}{15} & *(white)\textcolor{lightgray}{14} & *(white) \textcolor{lightgray}{10} & \none[3] \\
  \none & *(yellow) 4 & *(green) 9 & *(green) 8 & *(yellow) 5 & \none[2] \\
  *(white) \textcolor{lightgray}{13}& *(green) 6 & *(white) \textcolor{lightgray}{12} & *(white) \textcolor{lightgray}{11}& *(green) 7 & \none[1] \\
  \none[2] & \none[3] & \none[4] & \none[5] & \none[6]\\
\end{ytableau}
}
&
\ \ \
&
\ytableausetup{aligntableaux=bottom}
{\Large
\begin{ytableau}
   \none & \none & \none & \none & *(orange) 3 & \none[5] \\
  \none & \none & \none & *(red) 4 & *(red) 4 & \none[4] \\
  \none & \none & *(white) \textcolor{lightgray}{0} & *(white)\textcolor{lightgray}{0} & *(white) \textcolor{lightgray}{0} & \none[3] \\
  \none & *(yellow) 2 & *(green) 1 & *(green) 1 & *(yellow) 2 & \none[2] \\
  *(white) \textcolor{lightgray}{0}& *(green) 1 & *(white) \textcolor{lightgray}0 & *(white) \textcolor{lightgray}{0}& *(green) 1 & \none[1] \\
  \none[2] & \none[3] & \none[4] & \none[5] & \none[6]\\
\end{ytableau}
}
\end{matrix}
$
    
     \caption{Left: A pipe dream with crosses labeled in the order we add them in weak Bruhat order. Middle: An inversions diagram with inversions labeled in the order the corresponding crosses are added, then extended to a balanced tableaux. Right: the corresponding inversions tableaux.}
        \label{fig:in proof}
    \end{figure}

    (IT2): The nonzero entries in column $i$ of $\phi(P)$ are the rows in which there is a crossing through which pipe $i$ passes vertically. Since pipe $i$ cannot pass vertically through two different pipes in the same row, the nonzero entries in column $i$ of $\phi(P)$ must be distinct. 

    (IT3'): Pipe $i$ cannot have any crossings below row $i$, so all entries in row $i$ of $\phi(P)$ are at most $i$. 

    To complete the proof, we now construct the inverse of $\phi$. Consider a tableau $T \in \IT(w)$. Since $T$ is weakly balanced, by Lemma \ref{lemma:balanced} we can assign a total order on boxes to get a balanced tableaux compatible with the numbering on $T$ (so unshaded boxes will have smaller content than shaded ones). Now reverse this order and list the indices for boxes in this new order as $(i_1,j_1),(i_2,j_2),\dots,(i_{l(w)},j_{l(w)}),$ followed by all unshaded boxes.
    
    Beginning from the identity, we need to construct a sequence of pipe dreams $P_0, P_1, \dots, P_{l(w)}$, where $P_0$ is the identity pipe dream, $P_m$ is a pipe dream for the permutation $w_m$ with inversion set $(i_1,j_1),\dots,(i_m,j_m)$, and where $P_{m}$ is obtained from $P_{m-1}$ by adding the crossing of pipes $i_{m}$ and $j_{m}$ in row $T(i_m,j_m)$.
    We will assume for induction that we have already constructed $P_1, \dots, P_m$, and show we can add the desired crossing to get $P_{m+1}$.
    
    Again by Proposition \ref{prop:balanced chain bij} we know that adding the inversions in this order gives us a chain in weak Bruhat order, so $w_{m+1}$ directly covers $w_m$. Thus in $P_m$, pipes $i_{m+1}$ and $j_{m+1}$ must be next to each other since crossing them should be a simple transposition. Both pipes start in or below row $T(i_{m+1},j_{m+1})$ by rule (IT3'). In $P_w$, if neither pipe $i_{m+1}$ nor $j_{m+1}$ cross a pipe of smaller index than itself in row $T(i_m,j_m)$, then after their last crossing they remain next to each other from row $T(i_m,j_m)$ up, and thus we can make them cross in any row $T(i_{m+1},j_{m+1}) \leq T(i_m,j_m)$.
    If at least one of them does cross a pipe of smaller index than itself in row $T(i_m,j_m)$, then by rule (IT2) we must have $T(i_{m+1},j_{m+1}) \leq T(i_m,j_m)-1$. Conveniently $P_w$ has no crossings above row $T(i_m,j_m)$, so pipes $i_{m+1}$ and $j_{m+1}$ will be next to each other for all possible rows $T(i_{m+1},j_{m+1})<T(i_m,j_m)$, and we can again cross them where we need. These two cases are illustrated in Figure \ref{fig:add crossing}.

 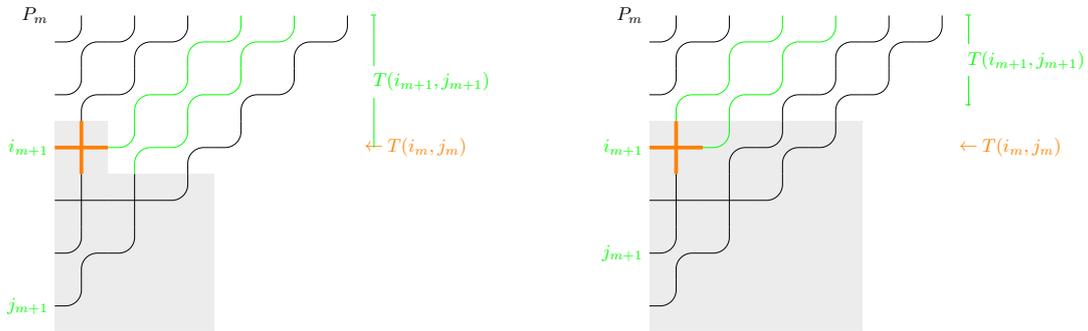
\begin{figure}[h]
        \centering
\scalebox{0.7}{
$
\begin{matrix}
\begin{tikzpicture}

\shadedsquares[gray!30]{3/1,
4/1,4/2,4/3,
5/1,5/2,5/3,
6/1,6/2,6/3}
\crosses{4/1,4/2}
\crosses[orange, line width=2pt]{3/1}
\bumps{1/1,1/2,2/1,5/1}
\pdcaps{1/6,2/5,3/4,4/3,5/2,6/1}
\bumps[green]{1/4,2/3,3/2}
\bumps[][green]{1/3,2/2}
\bumps[green][]{1/5,2/4,3/3}

\node[left] at (0,0) {$P_m$};
\node[left,green] at (0,-2.5) {$i_{m+1}$};
\node[left,green] at (0,-5.5) {$j_{m+1}$};
\node[green] at (6,0) {-};
\node[orange] at (6,-2.5) {$\leftarrow$};
\node[orange] at (7,-2.5) {$T(i_m,j_m)$};
\draw[green] (6,0) -- (6,-2.5) node [midway, fill=white, right=-4pt] {$T(i_{m+1},j_{m+1})$};

\end{tikzpicture}
&
\ \ \ \ \ \ \ \ \ \
&
\begin{tikzpicture}

\shadedsquares[gray!30]{3/1, 3/2, 3/3, 3/4,
4/1,4/2,4/3, 4/4,
5/1,5/2,5/3, 5/4,
6/1,6/2,6/3, 6/4}
\crosses{4/1,4/2}
\crosses[orange, line width=2pt]{3/1}
\bumps{1/1,5/1,1/5,2/4,3/3}
\pdcaps{1/6,2/5,3/4,4/3,5/2,6/1}
\bumps[green][]{1/4,2/3,3/2}
\bumps[green]{1/3,2/2}
\bumps[][green]{1/2,2/1}

\node[left] at (0,0) {$P_m$};
\node[left,green] at (0,-2.5) {$i_{m+1}$};
\node[left,green] at (0,-4.5) {$j_{m+1}$};
\node[green] at (6,0) {-};
\node[orange] at (6,-2.5) {$\leftarrow$};
\node[orange] at (7,-2.5) {$T(i_m,j_m)$};
\node[green] at (6,-1.7) {-};
\draw[green] (6,0) -- (6,-1.7) node [ midway, fill=white, right=-4pt] {$T(i_{m+1},j_{m+1})$};

\end{tikzpicture}
\end{matrix}
$
}
    
        \caption{Given the beginning of the pipe dream up to the first $m$ crossings as shaded, the green pieces of pipe indicate locations where the $(m+1)$st crossing can occur in some row $T(i_{m+1},j_{m+1})$.}
        \label{fig:add crossing}
    \end{figure}

    We note that this map is well-defined across multiple choices of total ordering; different paths from the identity to $w$ in weak Bruhat order can only differ by changing the order according to the relations $s_i s_j =s_j s_i$ for $|i-j| \leq 2$ or $s_i s_{i+1} s_i = s_{i+1} s_i s_{i+1}$. In the first case, this corresponds to crossing two completely separate pairs of pipes which we can do in either order with the same result. We will default to adding the leftmost crossing first in all cases. In the second case, all these crossings occur among the same three pipes $h<k<l$. Rule (IT2) mandates that $T(h,l) \neq T(k,l)$, so there can only be one compatible total ordering on these three crossings.
\end{proof}

\begin{remark}
  Starting with an inversions tableaux $T$ for a permutation $w$, this inverse bijection also allows us to obtain the $a$-compatible sequence (as defined in \cite{bergeron1993rc}) corresponding to $T$. When doing the inverse map in the proof above, instead of drawing crossings, record the pair $(a_m, T(i_m,j_m))$ at each step, where $a_m$ is the index of the simple transposition such that $w_m = s_{a_m} w_{m-1}$. Then $a = a_{l(w)} \dots a_2 a_1$ is a reduced word for $w$, and $\alpha=(T(i_{l(w)},j_{l(w)}), \dots, T(i_2,j_2), T(i_1,j_1))$ is an $a$-compatible sequence. The map $\Theta(T): T \mapsto (a,\alpha)$ is a bijection from inversions tableaux to $a$-compatible sequences for $w$.
\end{remark}

\begin{remark}
   The above proof also leads to bijections between balanced labellings and both inversions tableaux and pipe dreams.
Given an inversions tableaux $T$ we obtain the corresponding balanced labelling $T'$ (defined in \cite{balancedlabellings}) by setting $T(i,j) = T'(i,w_j)$ for all $i<j$ such that $(i,j) \in \Inv(w)$. It turns out that this is in fact the composition of the bijection $\Theta$ above with the decoding rule from \cite[Proof of theorem 4.3]{balancedlabellings}. This together with the proof of Theorem \ref{thrm:schub formula} further implies a direct bijection from pipe dreams to balanced labellings. Namely, the pipe dream $P\in \RP(w)$ corresponds to the balanced labelling $T$ in which $T(i,w_j)=r_{ij}$ for all $(i,j) \in \Inv(w)$ with $i<j$.
\end{remark}

By removing condition (IT3) and only keeping the first two, we can also get Stanley symmetric functions. They were originally defined by Stanley in \cite{stanleysymmetric} as a sum of certain quasi-symmetric function. However they have numerous equivalent definitions, including a definition as the stable limit of Schubert polynomials which we state below.

For a permutation $w = w_1 \dots w_m$ in one line notation, let $1^m \times w = 12\dots n (w_1+n) (w_2+n) \dots (w_m+n)$.

\begin{definition}
    The \emph{Stanley symmetric function} for the permutation $w$ is
    \[ F_w(x_1, x_2, \dots ) = \lim_{n\to\infty} \mathfrak{S}_{1^n \times w} (x), \]
    where each side is treated as a formal power series and we take the limit of each coefficient.
\end{definition}

In terms of inversions tableaux, the $n$'th term of this  limit corresponds to taking the limit as the shape at the top of the tableaux stays the same while we add $n$ more empty rows underneath corresponding to the $1^n$ at the beginning of our permutation. Ignoring the empty rows at the bottom, this is the equivalent of taking an inversions tableaux for $w$ but where the row and column indices begin at $n+1$ and $n+2$ respectively. In the limit as $n$ goes to infinity, we have all tableaux of shape $w$ which are weakly balanced and have no repeat entries in a column, but whose entries no longer need to satisfy an upper bound. We call these unbounded inversions tableaux. These are the equivalent of \emph{semistandard inversion fillings} in \cite{Kelly.thesis}. 

\begin{definition}
    An \emph{unbounded inversions tableau} for a permutation $w$ is an inversions diagram of shape $w$ with a positive integer entry in each shaded box (and all unshaded boxes are assumed to be 0) which satisfy rules (IT1) and (IT2). We denote the set of unbounded inversions tableaux for $w$ as $\mathcal{UIT}(w)$.
\end{definition}

Theorem \ref{thrm:schub formula} then implies the following formula for Stanley symmetric functions.

\begin{corollary}
    The Stanley symmetric polynomial for the permutation $w$ is given by 
    $$F_w = \sum\limits_{T \in \mathcal{UTI}(w)} x^{wt(T)}.$$
\end{corollary}

\section{Lehmer Code}
\label{section:lehmer code}

The Lehmer code of a permutation identifies the permutation and tells us useful information about Schubert polynomials -- it gives the maximal monomial in each Schubert polynomial and is key to showing that Schubert polynomials (over $S_{\infty}$) are a basis for the polynomial algebra $\mathbb{C}[x_1,x_2, \dots ]$. Inversions tableaux give a particularly nice way of viewing the Lehmer code and resulting impacts on the Schubert polynomial which we cover in this section.

\begin{definition}
    The \emph{Lehmer code} of a permutation $w$, written in one line notation, is $code(w) = (c_1, \dots, c_{n-1})$ where $c_i = \#\{ j > i | (i,j) \in \Inv(w) \}.$
\end{definition}
Later in this paper, we will also find it useful to think of $c_i$ as given by
$$c_i = \#\{n | n < w_i \text{ but } n \notin w_1, \dots ,w_{i-1}\}$$

Inversion diagrams provide an extremely convenient way to read of the Lehmer code -- $c_i$ is simply the number of shaded boxes in row $I$ of the inversions diagram $T_w$. Further, they provide a very nice interpretation for the lexicographically maximal monomial of a Schubert polynomial. 

The following result is well-known, e.g. it follows directly from \cite[Theorem 3.7]{bergeron1993rc}.

\begin{proposition}
  The lexicographically maximal monomial in the Schubert polynomial $\mathfrak{S}(w)$ is $x^{code(w)}$, and it has coefficient $1$.
\end{proposition}

Inversions tableaux make the proof of this result particularly simple.

\begin{proof}
    For a permutation $w$, consider the inversions tableaux where we fill each shaded box with the row number of that box. This is the unique lexicographically highest filling we can have satisfying rule (IT3'). Since within a row all values are the same, but columns have different values, this always satisfies rules (IT1) and (IT2). The content is exactly $code(w)$.
\end{proof}

In Example \ref{example:inversionstableaux}, the leftmost inversions diagram and monomial are lexicographically maximal.
The simplest case we can have, which sheds intuition on the behavior of these objects, is when this maximal monomial is the only monomial in $\mathfrak{S}_w$. This will happen when $w$ is a dominant permutation. A $\emph{dominant}$ permutation is any permutation whose Lehmer code is weakly decreasing. Equivalently, a dominant permutation can be described in terms of pattern avoidance. We say a permutation avoids the pattern 132 if there is no $i<j<k$ such that $w_i < w_k < w_j$. A permutation is dominant if and only if it is 132 avoiding. We completely characterize the set of inversions tableaux for dominant permutations.
 

\begin{lemma}
\label{lemma:downward closed}
    An inversions diagram is \emph{downward closed} if given any shaded box in a column, all boxes below it are shaded as well. The set of dominant permutations is exactly the set of permutations with downward closed inversions diagrams.
\end{lemma}

\begin{proof}
    We will use the 132 avoiding definition for dominant permutations. Suppose that a permutation $w$ contains a 132. This means we have some $i < j < k$ such that $w_i < w_k < w_j$. Of the three potential inversions $(i,j), \ (i,k)$ and $(j,k)$, this contains only the inversion $(j,k)$. Thus $w$ avoids having a 132 if and only if for every $(j,k) \in \Inv(w)$, we also have $(i,k) \in \Inv(w)$ for all $i < j < k$. In other words, in the Rectangle Rule, if the top right corner of the rectangle is shaded, the base box must be as well.
\end{proof}

\begin{example}
The inversions tableau for the dominant permutation 867435912.
\begin{eqnarray*}
\begin{ytableau}
\none & \none & \none & \none  & \none  & \none & \none & *(white) & \none[8] \\
\none & \none & \none & \none  & \none  & \none & *(green) 7 & *(white) & \none[7] \\
\none & \none & \none & \none  & \none  & *(white) & *(green)6 & *(white) & \none[6] \\
\none & \none & \none & \none  & *(green) 5 & *(green) 5 & *(green) 5 & *(white) & \none[5] \\
\none & \none & \none & *(white)  & *(green) 4 & *(green) 4 & *(green) 4 & *(white) & \none[4] \\
\none & \none & *(green) 3 & *(white)  & *(green) 3 & *(green) 3 & *(green) 3 & *(white) & \none[3] \\
\none & *(green) 2 & *(green) 2 & *(green) 2 & *(green) 2 & *(green) 2 & *(green) 2 & *(green) 2 & \none[2] \\
*(white) & *(green) 1 & *(green) 1 & *(green) 1 & *(green) 1 & *(green) 1 & *(green) 1 & *(green) 1 & \none[1] \\
\none[2] & \none[3] & \none[4] & \none[5]  & \none[6]  & \none[7] & \none[8] & \none[9] \\
\end{ytableau}
\end{eqnarray*}
\end{example}

The following well known formula (see 
\cite[P.6]{bergeron}) follows almost immediately.

\begin{corollary}
    For a dominant permutation $w$, the Schubert polynomial $\mathfrak{S}_w = x^{code(w)}$.
\end{corollary}

\begin{proof}
    By rule (IT2) the number in each shaded box of an inversions tableaux $T$ for $w$ is at most the row number, and by (IT3) it must be exactly the row number so that all entries in the column are distinct. Thus $T$ is unique and $\text{content}(T)=code(w)$.
\end{proof}


Now, we provide a precise description of which shaded tableaux are inversion diagrams of a dominant permutation (and obtain the unique inversions tableaux for the permutation by putting an $i$ in all shaded boxes in the $i^{th}$ row).

\begin{proposition}
    Let $T$ be a staircase tableaux with some boxes shaded. Then $T$ is the inversions diagram of some dominant permutation if and only if for all boxes $(i,j) \in T$ such that $(i,j)$ is the top shaded box of its column, all the boxes in the rectangle with corners $(i,i+1), (i,j),(1,i+1), (1,j)$ are shaded.
\end{proposition}

\begin{proof}
Let $B=(i,j)$ be the uppermost shaded box in column $j$. Then for any box $A$ left of $B$ in row $i$, boxes $A$ and $B$ form a rectangle with some unshaded box $C$ in column $j$ above $B$ as shown in Figure \ref{fig:dominant tableaux rectangle}. Since $B$ is shaded but $C$ is not, Corollary \ref{cor:inversion shape} tells us that box $A$ must be shaded. By the downwards closed rule from Lemma \ref{lemma:downward closed}, the rest of the boxes below box $A$ must be shaded as well. Thus all boxes of the rectangle with corners $(i,i+1), B=(i,j),(1,i+1), (1,j)$ are shaded. By starting with $B$ and shading more boxes in this way, we have guaranteed that both downward closure and the requirements of Corollary \ref{cor:inversion shape} for being a valid inversion set are met, so all staircase tableaux with boxes shaded according to this rule will indeed be valid inversion sets for dominant permutations.
\end{proof}

\begin{figure}[h]
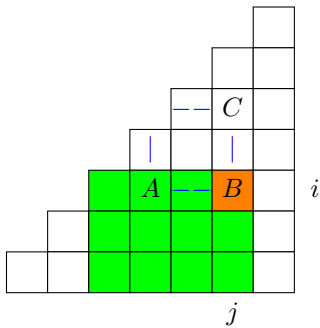

    \centering
    \ytableausetup{baseline}
\begin{ytableau}
\none & \none & \none  & \none  & \none & \none & *(white)\\
\none & \none & \none  & \none  & \none & *(white)  & *(white) \\
\none & \none & \none  & \none  & *(white) \textcolor{blue}{--} & *(white) C & *(white)\\
\none & \none & \none  & *(white) \textcolor{blue}{|} & *(white) & *(white) \textcolor{blue}{|} & *(white) \\
\none & \none & *(green)  & *(green) A & *(green) \textcolor{blue}{--} & *(orange) B & *(white) & \none[i]\\
\none & *(white) & *(green)  & *(green)  & *(green) & *(green) & *(white) \\
*(white) & *(white) & *(green)  & *(green)  & *(green) & *(green) & *(white) \\
\none & \none & \none & \none & \none & \none[j] & \none & \none\\
\end{ytableau}
    \caption{Given box $B=(i,j)$ as the uppermost shaded box in column $j$, all the green boxes in the green rectangle must also be shaded to be a valid inversions tableau for a dominant permutation.}
    \label{fig:dominant tableaux rectangle}
\end{figure}

For any permutation, we can also nicely describe the lexicographically minimal polynomial in $\mathfrak{S}_w$ using inversions tableaux. For this we will need a slight variation on the usual Lehmer code. 

\begin{definition}
    The \emph{column-Lehmer code} of $w$ is $ccode(w) = (c_1, \dots, c_n)$ where 
    $$c_i =  \# \{j < i | (i,j) \in \Inv(w)\} = \text{ the number of shaded boxes in column $i$ of } T_w.$$ 
\end{definition}

In \cite[\S 3]{bergeron1993rc} 
 Bergeron and Billey give a description of the lexicographically minimal pipe dream (called an RC-graph in the paper) for a permutation. Its definition, together with their Theorem 3.7, implicitly results in the following proposition. We can quickly give a direct proof of the formula using inversions tableaux.

\begin{proposition}
    For a permutation $w \in S_n$, let $ccode(w) = (c_1, \dots ,c_n)$. Then the lexicographically minimal monomial in $\mathfrak{S}_w$ is $\prod\limits_{i=1}^n x_i^{\#\{j | c_j \geq i \}}$ and it has coefficient 1.
\end{proposition}
    
 This minimal term corresponds to the inversions tableaux where we fill the shaded boxes in each column from the bottom to the top with the smallest number that has not yet been used in that column. For an example, see the rightmost inversions tableaux in Example \ref{example:inversionstableaux}.

\begin{proof}
    Let $T$ be a tableaux as above. By (IT2) we cannot get anything lexicographically lower in any column, so it remains to show that such $T$ satisfies (IT1) making it a valid inversions tableaux. Consider any triple of boxes $\textcolor{lightgray}{(j,k)}, \textcolor{green}{(i,k)} , \textcolor{Green}{(i,j)}$ with entries $x,y,z$ respectively. If $y=z$, these boxes automatically satisfy the rectangle rule. If $z>y$, there is some row $i'$ below $i$ such that $\textcolor{orange}{(i',j)}$ is shaded while $\textcolor{yellow}{(i',k)}$ is unshaded. Then by the Rectangle Rule, $\textcolor{lightgray}{(j,k)}$ is unshaded too, so $z>y>x=0$. If instead $z<y$, then by a similar argument $\textcolor{lightgray}{(j,k)}$ is shaded so $z<y<x$, meaning the Rectangle Rule is always upheld on the three boxes in question.

\begin{figure}[h]
    \centering
    \begin{tikzpicture}
\draw (0, 0.5) -- (0,1) -- (0.5,1) -- (0.5,1.5) -- (1,1.5) -- (1, 2) -- (1.5, 2) -- (1.5, 2.5) -- (2, 2.5) -- (2, 3)  -- (2.5, 3) -- (2.5, 3.5) -- (3, 3.5) ;
\draw[fill=gray!30]   (2,2.5) rectangle (2.5,2) node[pos=.5] {$x$};
\draw[fill=Green]   (1,1) rectangle (1.5,.5) node[pos=.5] {$z$};
\draw[fill=green]   (2,1) rectangle (2.5,.5) node[pos=.5] {$y$};
\draw[fill=orange]   (1,0) rectangle (1.5,-.5);
\draw[fill=yellow]   (2,0) rectangle (2.5,-.5);
\draw[dashed, blue] (1.25, 1) -- (1.25, 2.25) -- (2, 2.25);
\draw[dashed, blue] (2.25, 1) -- (2.25, 2);
\draw[dashed, blue] (1.25, 0.5) -- (1.25, 0);
\draw[dashed, blue] (2.25, 0.5) -- (2.25, 0);
\draw[dashed, blue] (1.5, .75) -- (2, .75);
\draw[dashed, blue] (1.5, -0.25) -- (2, -0.25);
\path (3.5,2.25) node  {$j$}
        (3.5,.75) node   {$i$}
        (3.5,-0.25) node {$i'$} 
        (1.25, -1) node {$j$}
        (2.25, -1) node {$k$};
\end{tikzpicture}
    \caption{The rectangle comparing boxes $\textcolor{lightgray}{(j,k)}, \textcolor{green}{(i,k)} , \textcolor{Green}{(i,j)}$ with entries $x,y,z$ respectively. For a row $i'$ below, boxes $\textcolor{orange}{(i',j)}$ and $\textcolor{yellow}{(i',k)}$ are also related to $\textcolor{lightgray}{(j,k)}$ by the Rectangle Rule.}
    \label{fig:double rectangle}
\end{figure}

To show the coefficient is 1, we prove there are no other permutations of elements within any columns of $T$ that still produce a valid inversions tableaux. Assume there is at least one column with entries out of order. Let $k$ be the leftmost such column, and $\textcolor{lightgray}{(j,k)}, \textcolor{green}{(i,k)}$ (still with entries $x,y$ respectively) be the lowest boxes in $k$ with entries out of order (i.e $x<y$). Then in Figure \ref{fig:double rectangle}, we have $z \geq y > x$. Since column $j$ is still in order, there must be some row $i'$ below $i$ such that $T(\textcolor{orange}{i',j}) = x$. But then we must have $T(\textcolor{orange}{i',j}) = T(\textcolor{yellow}{i',k}) = T(\textcolor{lightgray}{j,k}) = x$ to satisfy the Rectangle Rule, which violates (IT2).
\end{proof}

\begin{remark}
   For all non-dominant permutations, the lexicographically maximal and minimal inversions tableaux we have described are different, so these two constructions together prove that $\mathfrak{S}_{w}$ is a monomial if and \emph{only if} $w$ is dominant.
\end{remark}

\section{Grassmannian and Inverse Grassmannian Case}
\label{section:Grassmannian}

Dominant permutations belong to a larger class called \emph{vexillary} permutations -- permutations which are 2143-avoiding. A permutation is vexillary if and only if its Schubert polynomial is a \emph{flagged Schur polynomial}. Lascoux and Schützenberger first defined flag Schur polynomials in \cite{LascouxS} in their study of Schubert polynomials, and Wachs provided an explicit expression for the flags in terms of the permutation \cite{Wachs}. Flag Schur polynomials are defined combinatorial as follows. 
Given a partition $\lambda = (\lambda_1 \geq \dots \geq \lambda_n \geq 0)$ and a weakly increasing sequence of positive integers $b=(b_1 \leq \dots \leq b_n)$, let $SSYT(\lambda,b)$ be the set of SSYT of shape $\lambda$ in which row $i$ has entries at most $b_i$. Then the corresponding flagged Schur polynomial is
$$s_{\lambda}(b) = \sum\limits_{T \in SSYT(\lambda, b)} x^{wt(T)}.$$
In this section, we discuss two types of vexillary permutations which have particularly nice inversions tableaux -- Grassmannian and inverse Grassmannian permutations. 

The Schubert polynomial $\mathfrak{S}_w$ is a Schur polynomial exactly when $w$ is a \emph{Grassmannian permutation} -- a permutation with exactly one descent. In particular, \emph{$k$-Grassmannian} permutations are those with descent in spot $k$ (i.e. $w_k>w_{k+1}$). The $k$-Grassmannian permutations in $S_n$ are in bijection with Young Diagrams which fit in a $k \times (n-k)$ rectangle. Specifically, the permutation $w = \textcolor{Green}{a_1 a_2 \cdots a_k} b_1 b_2 \cdots b_{n-k}$ corresponds to the Young diagram $\lambda_w = (a_k - k, \dots, a_2 -2, a_1-1)$.

\begin{proposition}
\cite[\S 2]{bergeron}
    Let $w$ be a $k$-Grassmannian permutation. Then $\mathfrak{S}_w(x)= s_{\lambda_w}(x_1, \dots , x_k).$
\end{proposition}

One would hope then, that given $w$ as above, there would be a nice correspondence between $\IT(w)$ and SSYT of shape $\lambda_w$. Indeed we will prove that the shaded squares of tableaux in $\IT(w)$ give exactly reverse SSYT of shape $\lambda(w)$. These are Young diagrams which (weakly) decrease as opposed to increase along rows and columns, but likewise give a generating function formula for Schur polynomials.

\begin{proposition}
\label{prop:reverse SSYT}
Given a $k$-Grassmannian permutation $w\in S_n$, the set $\IT(w)$ consists of all tableaux whose shaded boxes form a reverse SSYT of shape $\lambda_w$ with top left corner in position $(k,k+1)$, and whose entries are in $1, \dots, k$. 
\end{proposition}

\begin{figure}[h]
    \centering
    \ytableausetup{baseline}
\begin{ytableau}
\none & \none & \none & \none  & \none  & \none & \none & *(white) & \none[8] \\
\none & \none & \none & \none  & \none  & \none & *(white) & *(white) & \none[7] \\
\none & \none & \none & \none  & \none  & *(white) & *(white) & *(white) & \none[6] \\
\none & \none & \none & \none  & *(white)  & *(white) & *(white) & *(white) & \none[5] \\
\none & \none & \none & *(green) 4 & *(green) 4  & *(green) 4 & *(green) 2 & *(green) 2 & \none[4] & \none[\textcolor{Green}{\lambda_1}]\\
\none & \none & *(white) & *(green) 3  & *(green) 3 & *(green) 1 & *(white) & *(white) & \none[3] & \none[\textcolor{Green}{\lambda_2}]\\
\none & *(white) & *(white) & *(green) 2 & *(green) 2 & *(white) & *(white) & *(white) & \none[2] & \none[\textcolor{Green}{\vdots}]\\
*(white) & *(white) & *(white) & *(green) 1 & *(green) 1 & *(white) & *(white) & *(white) & \none[1] & \none[\textcolor{Green}{\lambda_k}]\\
\none[2] & \none[3] & \none[4] & \none[5]  & \none[6]  & \none[7] & \none[8] & \none[9] \\
\end{ytableau}
    \caption{An inversions tableaux of shape $w = \textcolor{Green}{a_1 a_2 \cdots a_k} b_1 b_2 \cdots b_{n-k} = \textcolor{Green}{3469}12578$. Its shaded boxes are a reverse SSYT of shape $\lambda_w=(5,3,2,2).$}
    \label{fig:IT346912578}
\end{figure}

\begin{proof}
    The shape follows immediately from recording which positions in the permutation have values that are pairwise out of order. For a tableaux $T \in \IT(w)$, the reverse semi-standard filling of the numbers is implied as follows by our three rules. First, consider any shaded boxes $A, B$ in the same row (or column), with $A$ on the left (or top). By the Rectangle Rule, $B$ is the base box, and the third box to compare to must be out of our $k \times (n-k)$ rectangle and therefor not shaded. To satisfy (IT1), Lemma \ref{lemma:balanced} then requires that $T(A) \geq T(B)$. Rule (IT2) ensures that the columns are strictly decreasing to avoid repeat entries. Finally rule (IT3) guarantees all the entries are at most $k$, since in the upper left corner $T(k,k+1) \leq k$.
\end{proof}

Grassmannian permutations $u < w$ in weak Bruhat order if the inversions diagram for $w$ contains the inversions diagram for $u$. In other words, they have the same top corner (so the descent is in the same spot $k$ for some $k$) and $\lambda_u \subset \lambda_w$. What's more, the only possible inversions tableaux in the $k \times (n-k)$ rectangle these tableaux sit in are those whose shaded boxes form a Young Diagram with top left corner in box $(k, k+1)$ --- if any shaded box $(i,j)$ had an unshaded box $(i, j-1)$ immediately left (or $(i+1, j)$ immediately above) of it within the rectangle, then we would be including only the middle box of the order triple $(j-1, j), (i,j), (i,j-1)$ , resulting in an invalid permutation shape. Thus our inversions diagrams embody the correspondence between weak Bruhat order on $k$-Grassmannian permutations and Young's lattice on tableaux that fit within a $k \times (n-k)$ rectangle.

Based on this structure, we can further consider the case of skew inversions tableaux for Grassmannian permutations. For $u < w$ both $k$-Grassmanian permutations related in weak Bruhat order, we can consider the diagram of skew shape $w/u$ obtained by shading all boxes in $\Inv(w) \setminus \Inv(u)$. Let $\mathcal{IT}(w/u)$ be the set consisting of all the tableaux where we put a $k$ in all the boxes in $\Inv(u)$, a zero in all other unshaded boxes, and fill shaded boxes with numbers satisfying rules (IT1),(IT2),(IT3). Such tableaux correspond exactly to all the reverse SSYT of skew shape $\lambda_w/\lambda_u$ with numbers $1, \dots, k$. Note that if we used rule (IT3') instead, then we would no longer get all reverse SSYT fillings.

\begin{figure}[H]
    \centering
    \ytableausetup{baseline}
\begin{ytableau}
\none & \none & \none & \none  & \none  & \none & \none & *(white) & \none[8] \\
\none & \none & \none & \none  & \none  & \none & *(white) & *(white) & \none[7] \\
\none & \none & \none & \none  & \none  & *(white) & *(white) & *(white) & \none[6] \\
\none & \none & \none & \none  & *(white)  & *(white) & *(white) & *(white) & \none[5] \\
\none & \none & \none & *(orange) & *(orange)  & *(orange) & *(green) 4 & *(white) & \none[4] \\
\none & \none & *(white) & *(orange)  & *(orange)  & *(green) 3 & *(green) 2 & *(white) & \none[3] \\
\none & *(white) & *(white) & *(green) 4 & *(green) 1 & *(green) 1 & *(white) & *(white) & \none[2] \\
*(white) & *(white) & *(white) & *(green) 2 & *(white) & *(white) & *(white) & *(white) & \none[1] \\
\none[2] & \none[3] & \none[4] & \none[5]  & \none[6]  & \none[7] & \none[8] & \none[9] \\
\end{ytableau}
    \caption{An inversions tableaux of skew shape $(\textcolor{Green}{2578}13469)/(\textcolor{Green}{1257}34689)$. Orange boxes (shaded, unnumbered) are treated as containing 4, while unshaded boxes are still 0.}
    \label{fig:skew Grassmannian}
\end{figure}

For $k$-Grassmannian permutations $u,w$ related in weak Bruhat order, define the \emph{tableaux skew Schubert polynomial} $$\mathfrak{G}_{w/u}(x) = \sum\limits_{T\in \mathcal{IT}(w/u)}x^{wt(T)}.$$ 
It follows that $\mathfrak{G}_{w/u}(x) = s_{\lambda_w / \lambda_u }(x_1, \dots, x_k)$. Note this is not the same as the skew Schubert polynomials defined in \cite{lenartsottile}. It is known that $s_{\lambda/\mu} = \sum\limits_{\nu} c^{\lambda}_{\mu,\nu} s_\nu$ where $c^{\lambda}_{\mu,\nu}$ are the \emph{Littlewood--Richardson coefficients} defined to satisfy $s_\mu s_\nu = \sum\limits_\lambda c^{\lambda}_{\mu,\nu} s_\lambda$ \cite{EC2}. In fact, this expression of skew Schur polynomials in terms of Littlewood--Richardson coefficients plays a key role in many proofs for combinatorial objects to count these coefficients. It also follows from the definition of Littlewood--Richardson coefficients that for $w,u,v$ all $k$-Grassmanian permutations we have $c^w_{u,v} = c^{\lambda_w}_{\lambda_u,\lambda_v}$. Putting these together, we can express our tableaux skew Schubert polynomial in terms of Schubert Littlewood--Richardson coefficients! Namely,
$$\mathfrak{S}_{w/u} = \sum\limits_{v} c^{w}_{u,v} s_v.$$

\begin{question}
    Is there a nice way to extend the definition of skew inversions tableaux and tableaux skew Schubert polynomials to any pair of permutations $u,w$ related in weak Bruhat order? In particular, will such polynomials be Schubert positive or expressible in terms of Schubert Littlewood--Richardson coefficients?
\end{question}


Changing gears slightly, we next consider inversions tableaux for inverse Grassmannian permutations.

\begin{proposition}
\label{prop:inverse Grassmannian}
    Let $w = \textcolor{Green}{a_1 a_2 \cdots a_k} b_1 b_2 \cdots b_{n-k}$ be a $k$-Grassmannian permutation. Then the inversions diagram of shape $w^{-1}$ has boxes only in intersections of columns $a_1, \dots, a_k$ and rows $b_1, \dots, b_{n-k}$. Removing empty rows and columns and rotating 180 degrees, we get the bijection
    $\IT(w^{-1}) \xrightarrow{\sim} SSYT(\lambda=(k+i-b_i)_{i=1}^{n-k}, (b_1, \dots , b_{n-k}))$.
\end{proposition}

\begin{figure}[h]
    \centering
    \ytableausetup{aligntableaux=center}
\begin{ytableau}
 \none & \none & \none  & \none  & \none & \none & *(white) & \none[7] \\
\none & \none & \none  & \none  & \none & *(green) 5 & *(white) & \none[6] \\
\none & \none & \none  & \none  & *(white) & *(white) & *(white) & \none[5] \\
\none & \none & \none  & *(white)  & *(white) & *(white) & *(white) & \none[4] & \none & \none[\longrightarrow] & \none \\
\none & \none & *(green) 3 & *(green) 3  & *(white) & *(green) 2 & *(white) & \none[3] \\
\none & *(white) & *(white)  & *(white)  & *(white) & *(white) & *(white) & \none[2] \\
*(green) 1 & *(white) & *(green) 1 & *(green) 1 & *(white) & *(green) 1 & *(white) & \none[1] \\
\none[2] & \none[3] & \none[4] & \none[5]  & \none[6]  & \none[7] & \none[8] \\
\end{ytableau}
\begin{ytableau}
    *(green) 1 & *(green) 1 & *(green) 1 & *(green) 1 & \none[\ \ \leq 1] \\
    *(green) 2 & *(green) 3 & *(green) 3 & \none & \none[\ \ \leq 3]\\
    *(green) 5 & \none & \none & \none & \none[\ \ \leq 6] \\
\end{ytableau}
    \caption{An inversions tableaux for $(\textcolor{Green}{2457}1368)^{-1}$ and its corresponding flag SSYT.}
    \label{fig:inverse Grassmannian}
\end{figure}

\begin{proof}
For a $k$-Grassmannian permutation $w = \textcolor{Green}{a_1 a_2 \cdots a_k} b_1 b_2 \cdots b_{n-k}$, the inversion set $\Inv(w^{-1})=\{ (a_i,b_j) | i<j \text{ but } a_i > b_j \}$. In other words, it is pairs of numbers (rather than positions) which are out of order in $w$. This restricts shaded boxes to the intersection of columns and rows given in Proposition \ref{prop:inverse Grassmannian}. Consider any shaded boxes $A, B$ in the same column (or row) with $A$ on the top (or left). Then $B$ is the base box in the rectangle rule, and the third box will always be unshaded, so we must be (weakly) decreasing along rows and down columns. Sliding shaded boxes together and rotating by 180 degrees, we then get a flag SSYT. the shape $\lambda$ from the proposition comes from the fact that $\lambda_i = \{j | a_j > b_i\}$.
\end{proof}

\section{Chute Moves Poset}
\label{section:chute moves}

Our goal in this section is to establish connections between inversions tableaux and Lehmer tableaux on the one hand and Rubey's chute move posets on the other hand. These connections will be needed in a forthcoming article with Defant, Mularczyk, Nguyen Van, and Tung devoted to proving that chute move posets are lattices. 

\begin{definition}
\label{def: chute move}
    In a pipe dream, a \emph{chute move} $C_{ij}$ on pipes $i$ and $j$ (with $i<j$) between rows $a$ and $b$  is a change of the following type, where pipes $i$ and $j$ bound a rectangle which is filled with crosses on the inside.
    
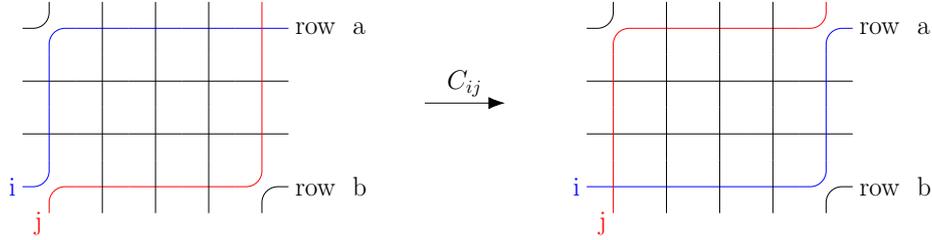
\begin{center}
\scalebox{0.7}{
$
\begin{matrix}
\begin{tikzpicture}
    \crosses[blue][]{1/2,1/3,1/4}
    \crosses[][blue]{2/1,3/1}
    \crosses[red][]{4/2,4/3,4/4}
    \crosses[][red]{2/5,3/5}
    \crosses[blue][red]{1/5}
    \crosses{2/2,2/3,2/4,3/2,3/3,3/4}
    \bumps[][blue]{1/1}
    \bumps[blue][red]{4/1}
    \bumps[red][]{4/5}
     \node[right] at (5,-0.5) {\Large \text{row } a};
    \node[right] at (5,-3.5) {\Large \text{row } b};
     \node[left] at (0,-3.5) {\Large \textcolor{blue}{i}};
    \node[left] at (0.5,-4.2) {\Large \textcolor{red}{j}};
\end{tikzpicture}
&
\begin{tikzpicture}
  \draw[-{Latex[length=3mm]}]    (0,1)   -- (1.5,1) node [above, midway] {\Large $C_{ij}$};
  \node[] at (-0.5,-1.5) {};
  \node[] at (2.2,0) {};
\end{tikzpicture}
&
\begin{tikzpicture}
    \crosses[red][]{1/2,1/3,1/4}
    \crosses[][red]{2/1,3/1}
    \crosses[blue][]{4/2,4/3,4/4}
    \crosses[][blue]{2/5,3/5}
    \crosses[blue][red]{4/1}
    \crosses{2/2,2/3,2/4,3/2,3/3,3/4}
    \bumps[][red]{1/1}
    \bumps[red][blue]{1/5}
    \bumps[blue][]{4/5}
     \node[right] at (5,-0.5) {\Large \text{row } a};
    \node[right] at (5,-3.5) {\Large \text{row } b};
     \node[left] at (0,-3.5) {\Large \textcolor{blue}{i}};
    \node[left] at (0.5,-4.2) {\Large \textcolor{red}{j}};
\end{tikzpicture}
\end{matrix}
$
}
\end{center}
\end{definition}

Note that in \cite{billey1993some} this is what Bergeron and Billey refer to as a generalized chute move.
Observe that if such a chute move exists for the given pipes $i,j$, then the values of $a$ and $b$ are uniquely determined. We now define the corresponding action of the chute move $C_{ij}$ on inversions tableaux.


\begin{definition}
\label{Def:chute move}
Let $T$ and $T'$ be inversions tableaux for a permutation $w$, with $T(i,j)=a$ which satisfy the following:

    \begin{enumerate}
        \item $T'(i,j)=b$, where $b$ is the smallest number satisfying $b>a$ and $b \notin \{T(r,j) | r<i \}$.
        \item For some values of $k>j$ we have $T(i,k) = a = T'(j,k)$ and $T(j,k) = b = T'(i,k)$.
        \item For all other pairs of indices $(h,l)$ not listed above, we have $T(h,l) = T'(h,l)$.
    \end{enumerate}

Then we say $T' = C_{ij} T$, and we call $C_{ij}$ a \emph{chute move} on $T$.
\end{definition}

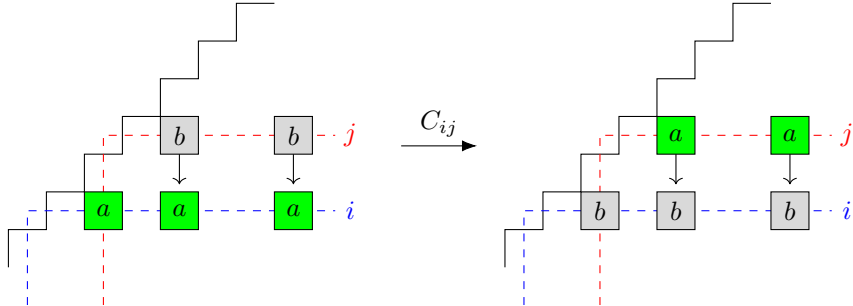
\begin{center}
$
\begin{matrix}
 \begin{tikzpicture}
\draw (0, 0.5) -- (0,1) -- (0.5,1) -- (0.5,1.5) -- (1,1.5) -- (1, 2) -- (1.5, 2) -- (1.5, 2.5) -- (2, 2.5) -- (2, 3)  -- (2.5, 3) -- (2.5, 3.5) -- (3, 3.5) -- (3,4) -- (3.5,4);
\draw[dashed, red] (1.25,0) -- (1.25, 2.25) -- (4.3, 2.25);
\draw[dashed, blue] (0.25,0) -- (0.25, 1.25) -- (4.3,1.25);
\draw[fill=green]   (1,1.5) rectangle (1.5,1) node[pos=.5] {$a$};
\draw[fill=gray!30]   (2,2.5) rectangle (2.5,2) node[pos=.5] {$b$};
\draw[fill=green]   (2,1.5) rectangle (2.5,1) node[pos=.5] {$a$};
\draw[fill=gray!30]   (3.5,2.5) rectangle (4,2) node[pos=.5] {$b$};
\draw[fill=green]   (3.5,1.5) rectangle (4,1) node[pos=.5] {$a$};
\path (4.5,2.25) node  {\textcolor{red}{$j$}}
        (4.5,1.25) node   {\textcolor{blue}{$i$}};
\draw[->]    (2.25,2)   -- (2.25,1.6);
\draw[->]    (3.75,2)   -- (3.75,1.6);
\end{tikzpicture}
&
\begin{tikzpicture}
  \draw[-{Latex[length=2mm]}]    (0,1)   -- (1,1) node [above, midway] {$C_{ij}$};
 \node[] at (0,-1) {};
\end{tikzpicture}
&
 \begin{tikzpicture}
\draw (0, 0.5) -- (0,1) -- (0.5,1) -- (0.5,1.5) -- (1,1.5) -- (1, 2) -- (1.5, 2) -- (1.5, 2.5) -- (2, 2.5) -- (2, 3)  -- (2.5, 3) -- (2.5, 3.5) -- (3, 3.5) -- (3,4) -- (3.5,4);
\draw[dashed, red] (1.25,0) -- (1.25, 2.25) -- (4.3, 2.25);
\draw[dashed, blue] (0.25,0) -- (0.25, 1.25) -- (4.3,1.25);
\draw[fill=gray!30]   (1,1.5) rectangle (1.5,1) node[pos=.5] {$b$};
\draw[fill=green]   (2,2.5) rectangle (2.5,2) node[pos=.5] {$a$};
\draw[fill=gray!30]   (2,1.5) rectangle (2.5,1) node[pos=.5] {$b$};
\draw[fill=green]   (3.5,2.5) rectangle (4,2) node[pos=.5] {$a$};
\draw[fill=gray!30]   (3.5,1.5) rectangle (4,1) node[pos=.5] {$b$};
\path (4.5,2.25) node  {\textcolor{red}{$j$}}
        (4.5,1.25) node   {\textcolor{blue}{$i$}};
        \draw[->]    (2.25,2)   -- (2.25,1.6);
\draw[->]    (3.75,2)   -- (3.75,1.6);
\end{tikzpicture}
\end{matrix}
$
\end{center}

We next show that chute moves on inversions tableaux correspond exactly to chute moves on pipe dreams. Recall from the proof of Theorem \ref{thrm:schub formula} that $\phi^{1}$ is the bijection from inversions tableaux to pipe dreams.

\begin{proposition}
    \label{Prop:chute move}
    Let $T$ be an inversions tableaux.
     If $C_{ij}$ is a chute move on $T$ then it will be on $\phi^{-1}(T)$ as well and vice versa. Further $\phi^{-1}(C_{ij}T)= C_{ij}(\phi^{-1}(T)).$
    
\end{proposition}

\begin{proof}
    First, given a chute move on pipe dreams as in Definition \ref{def: chute move}, we can see that the corresponding inversions tableaux satisfy all the conditions of Proposition \ref{Prop:chute move}. Next, we show the other direction. Assume $T$ and $T'$ are inversions tableaux which satisfy the conditions of Proposition \ref{Prop:chute move}. Then we know that in $\phi^{-1}(T)$, pipes $i$ and $j$ cross in row $a$. 
    
    We will show that there must also be a bump tile containing pipes $i$ and $j$ in column $b$. Assume for contradiction that this is not the case. Let $c_1$ be the largest column index of any tile containing pipe $i$ in row $b$, and let $c_2$ be the smallest column index of any tile containing pipe $j$ in row $b$. Since $i$ and $j$ have not crossed yet and do not share a bump tile in row $b$, we must have $c_1 < c_2$.

    By our construction, $\phi^{-1}(T)$ and $\phi^{-1}(T')$ must have all the same crossings below row $b$, so they agree below row $b$. Suppose pipe $k$ enters row $b$ in some column $c \leq c_2$. Then it cannot cross pipe $j$ vertically in row $b$, so by 3, we must have $T(l,j)=T'(l,j)$ for all $l$ both greater or less than $j$. From this, we know that $\phi^{-1}(T)$ and $\phi^{-1}(T')$ agree on all tiles in $\{(r,c)| r<b \} \cup \{(b,c)| c \leq c_2 \}$. However this contradicts pipes $i$ and $j$ being able to cross in row $b$ of $T'$.

    Now we know there must be a bump tile containing pipes $i$ and $j$ in row $b$. We also know from condition 2 that pipe $j$ must vertically cross some other pipe in every row between $a+1, a+2, \dots, b-1$. Thus pipe $j$ can have no bumps in rows $a+1, \dots b-1$. Every pipe that crosses $j$ horizontally in rows $a+1, \dots b-1$ must also cross pipe $i$ horizontally in one of these rows (otherwise it would have to cross wither $i$ or $j$ twice), so pipe $i$ likewise has no  bumps in rows $a+1, \dots b-1$. This gaurantees exactly the rectangular grid shape in Definition \ref{def: chute move}. On such a rectangle, making the changes as governed by conditions 1 and 2 is exactly a chute move.
\end{proof}
\begin{figure}
        \centering
    \begin{tikzpicture}
    \crosses[white][dashed]{1/3}
    \bumps[white][blue]{1/4, 2/2}
    \crosses[blue][white]{1/5}
    \crosses[dashed][white]{2/5}
    \crosses[red][white]{3/5}
    \crosses[blue][dashed]{2/3}
    \bumps[blue][black,dashed]{2/4}
    \crosses[black,dashed][blue]{3/2}
    \crosses[black,dashed]{3/3}
    \bumps[black,dashed][red]{3/4}
    \bumps[red][white]{3/6}
    \crosses[black,dashed][red]{2/6}
    \crosses[blue][red]{1/6}
    \bumps[blue][white]{4/2}
    \shadedsquares[gray!30]{3/1, 4/1, 4/2, 4/3, 4/4, 4/5}
     \node[right] at (6,-0.5) {\Large \text{row } a};
    \node[right] at (6,-2.4) {\Large \text{row } b};
     \node[left] at (1,-3.5) {\Large \textcolor{blue}{i}};
    \node[below] at (3.5,-3) {\Large \textcolor{red}{j}};
\end{tikzpicture}
        \caption{A region of the pipe dream corresponding to an inversions tableau T, in which pipes $i$ and $j$ do not share a bump tile in row $b$. Given $T$ and $T'$ as in the statement of Proposition \ref{Prop:chute move}, $\phi^{-1}(T)$ and $\phi^{-1}(T')$ must agree on all tiles weakly below left of the shaded region. They must also agree on all crossings of pipes that enter row $b$ left of pipe $j$ (illustrated with dashed lines). Note that this prevents $i$ and $j$ from intersecting in row $b$ of $\phi^{-1}(T')$.}
        \label{fig:IT-bump}
    \end{figure}
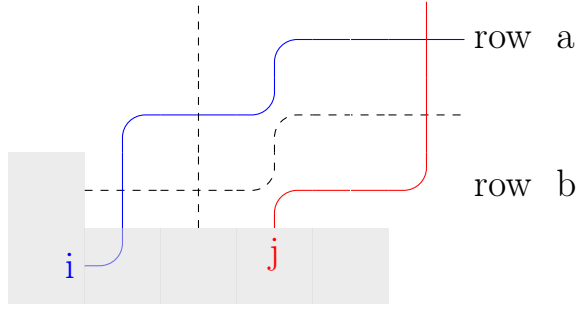

\begin{corollary}
\label{cor:chute on column}
Let $T$ be an inversions tableaux and $T'=C_{ij}T$ be the inversions tableaux obtained from $T$ by the chute move $C_{ij}$. Let $k$ be a column in which $T(i,k)=a$ and $T'(i,k)=b$. Then for all $x$ such that $a<x<b$, there exists some $i_x < i$ such that $T(i(x),k) = m$.
   \end{corollary}

   \begin{proof}
    This is true when $k=j$ by condition 1 of Proposition \ref{Prop:chute move}. For $x$ such that $a<x<b$, we then have $T(i_x,j) = m$ for some $i_x <i$. For any other column $i_x>j$ such that $T(i,k)=a$ and $T'(i,k)=b$, we have $T(j,k)=b$ and $T'(j,k)=a$.
    Since $T$ is an inversions tableaux, the rectangle rule tells us that $T(i_x,j) = m \leq T(i_x,k) < T(j,k) = b$ , so $T(i_x,k) \geq x$. Also $T'$ being an inversions tableaux tells us that $T(i_x,k) = x \geq T(i_x,k) > T(j,k) = a$, so in fact $T(i_x,k)=x$. Thus for all $x$ such that $a < x < b$, there is a box containing $x$ below $(i,k)$ in column $k$ as desired.
\end{proof}

Corollary \ref{cor:chute on column} then gives us a nice way to understand the action of a chute move as restricted to a single column.

Let $D$ be an inversions diagram for $w$, and $T$ be a tableaux with numbers in all the shaded boxes of $D$. Given a column $c$ of $T$, let $i_1 < i_2 \dots < i_m$ be all the row indices such that $(i_j,c) \in \Inv(w)$ for $j = 1, \dots, m$. In other words, the set of shaded boxes in column $c$ of $T$ is exactly $(i_1,c), \dots, (i_m,c)$. Define
$$word_c(T) = T(i_1,c) T(i_2,c) \dots T(i_m,c).$$
In the case where $T$ is an inversions tableau, $word_c(T)$ is a \emph{partial permutation}, that is, a word consisting of numbers $1, \dots, n$, where each number appears at most once.

\begin{example}
    For the inversions tableaux $T$ on the right in Figure \ref{fig:in proof}, we have $word_6(T)=1243$. For the tableaux $T$ in Figure \ref{fig:inverse Grassmannian}, we have $word_7(T)=125$.
\end{example}

\begin{observation}
\label{obs:on column}
A chute move that changes column $c$ of an inversions tableaux $T$ will do one of the following to the corresponding partial permutation $word_c(T) = u_1 u_2 \dots u_m$:
\begin{enumerate}[i)]
    \item For some $i$, replace $u_i$ with a number $b>u_i$ such that $b \notin \{u_1, \dots, u_m\}$, but for all $u_i < x < b$ there is some $j<i$ such that $u_j = x$.(This case happens when column $c$ is the leftmost column changed by the chute move).
    \item For some $i<j$, trade the values $u_i$ and $u_j$. This can only happen if $u_i<u_j$, and for all $u_i<x<u_j$ there is some $k<i$ such that $u_j = x$. (This case occurs for any other column $c$ that changes under the action of the chute move.)
\end{enumerate}
\end{observation}

We can simplify this action to a single rule by extending $word_c(T)$ to be a full permutation. Given a partial permutation $u = u_1 u_2 \dots u_m$ with all $u_i \leq n$, let $s_n(u) = u_1 u_2 \dots u_m \textcolor{Green}{\dots u_n}$ be the permutation in $S_n$ obtained from $u$ by tacking on all values in
$\{1, \dots, n\} \setminus \{u_1, \dots, u_m\}$ to the end of $u$ in increasing order.

\begin{example}
 For the inversions tableaux $T$ on the right in Figure \ref{fig:in proof}, we have $s_7(word_6(T))=1243\textcolor{Green}{567}$.
For the tableaux $T$ in Figure \ref{fig:inverse Grassmannian}, we have $s_8(word_7(T)) = 125\textcolor{Green}{34678}$. 
\end{example}

\begin{observation}
\label{obs:mediocre bruhat}
    On any column $c$ affected by a chute move, the chute move will always act on the corresponding word $s_n(word_c(T))=u_1 u_2 \dots u_n \in S_n$ by property ii) from Observation \ref{obs:on column}. This motivates our definition of Mediocre Bruhat order in the next subsection.
\end{observation}

\subsection{Mediocre Bruhat Order}
For a permutation $w=w_1 w_2\dots w_n$ written in one line notation, we say the operator $t_{ij}$ swaps the positions of the number $i$ and $j$. We will define a partial order on permutations which we relate to both their Lehmer codes and to the chute moves poset.

\begin{definition}
\label{def:mediocre bruhat}
    The \emph{mediocre Bruhat order} is a partial order on the set of permutations $S_n$ with covering relations given by
    $u=u_1u_2\dots u_n \lessdot w$ if all of the following hold:
    \begin{itemize}
        \item $w = u t_{ij}$ for some $i<j$,
        \item $u_i < u_j$, and
        \item For all $x$ satisfying $u_i<x<u_j$, there exists some $k<i$ such that $u_k$=x.
    \end{itemize}
\end{definition}

By Theorem \ref{Thrm: mediocre Bruhat}, the mediocre Bruhat order is isomorphic to the \emph{middle order} on permutations defined in \cite{middleorder} though they are not identical due to different choices of convention (much like left and right weak Bruhat order). Theorem 1.9 of \cite{middleorder} also provides a classification of the covering relations similar to our Definition \ref{def:mediocre bruhat} in terms of permutation matrices. 

Observe that for a fixed choice of $u$ and $i$, there always exists a unique $j$ such that $t_{ij}u$ covers $u$ in mediocre Bruhat order as per the properties above.

\begin{example}
  The permutation $ u = 5236\textcolor{Green}{4}1\textcolor{Green}{7} \lessdot 5236\textcolor{Green}{7}1\textcolor{Green}{4} = t_{57} u$ because both 5 and 6 fall to the left of $\textcolor{Green}{4}$ in one line notation for $u$.

  However 23\textcolor{Green}{5}4\textcolor{Green}{7}16 is not covered by 23\textcolor{Green}{7}4\textcolor{Green}{5}16 because $\textcolor{Green}{5}< 6 < \textcolor{Green}{7}$, but $6$ falls to the right of $5$ in one line notation of $u$.
\end{example}

Note that if two permutations $u$ and $w$ are related in weak Bruhat order, then they will be related in mediocre Bruhat order. If they are related in mediocre Bruhat order, they will be related in strong Bruhat order, so mediocre Bruhat order is strictly in between strong and weak Bruhat order.

Observation \ref{obs:mediocre bruhat} directly implies the following Lemma.

\begin{lemma}
\label{lemma:extended column}
    Let $T$ be an inversions tableau for a permutation $w \in S_n$. Let $T'$ be obtained from $T$ via chute move. Then for all columns $c$, we have either $s_n(word_c(T)) = s_n(word_c(T'))$ or $s_n(word_c(T)) \lessdot s_n(word_c(T'))$ in mediocre Bruhat order.
\end{lemma}

The natural next question is, given permutations $u$ and $w$ in $S_n$, how do we know if they are related in mediocre Bruhat order? The answer is in their Lehmer codes.

\begin{theorem}
\label{Thrm: mediocre Bruhat}
    Let $u,w \in S_n$ be permutations with Lehmer codes $(c_1, \dots , c_n)$ and $(d_1, \dots , d_n)$ respectively. Then $u \leq w$ in mediocre Bruhat order if and only if $c_i \leq d_i$ for all $i$.
\end{theorem}

\begin{proof}
    For Lehmer codes $c,d$ if $c_i \leq d_i$ for all $i$, denote this by $c \leq_L d$. This defines a partial order on Lehmer codes. This order has covering relation $c \lessdot_L d$ if and only if $d = c + e_i$ for some standard basis vector $e_i$ (the vector with a one in spot $i$ and zeros elsewhere). To prove the Theorem, it suffices to show that for permutations $w,u \in S_n$, we have $u \lessdot w$ if and only if $code(u) \lessdot_L code(w)$.  

    We will actually show this in $S_{\infty}$ where we can always add any $e_i$ to go up in the Lehmer code poset. To see for $S_n$, we simply restrict to permutations (or Lehmer codes of permtuations) in $S_n$.

    Let $w=w_1 w_2 w_3 \dots \in S_{\infty}$ be a permtuation with Lehmer code $c=(c_1,c_2,c_3,\dots)$. For any $i \in Z_{\geq 0}$, let $w' = t_{ij}w$ where $w_j$ is the smallest value in the set $\{w_k | k>i\}$. This is equivalent to saying $w'$ covers $w$ in mediocre Bruhat order (and picking a specific $i$), so it is enough to show that $code(w') = c + e_i$. Let $code(w') = c' = (c'_1, c'_2, \dots)$. Since $w_j = w'_j$ for all $j<i$, we also have that $c_j = c'_j$ for all $j<i$. Next, we know $w'_i =w_j$ where $w_j$ is the smallest value greater than $w_i$ that is not contained in $\{w_1,w_2, \dots, w_i\}$. This guarantees that
    $$c'_i = \#\{n | n < w'_i \text{ but } n \notin w_1, \dots ,w_{i-1}\} = \#\{n | n < w_i \text{ but } n \notin w_1, \dots ,w_{i-1}\}+1 =  c_i+1.$$
    
    For all other spots $k$ such that $k > i$ and $k\neq j$, we have that $w_k < w_i$ or $>w_j$, and as a result the number of entries smaller than $w_k=w'_k$ is the same in the sets $\{w_1, \dots , w_{k-1}\}$ and $\{w'_1, \dots, w'_{k-1}\}$, so $c_k = c'_k$.
    Finally for spot $j$ we have that $\{w_1, \dots ,w_j\} = \{w'_1, \dots , w'_j\}$ and also as sets
   $\{n | n < w_j \text{ but } n \notin w_1, \dots ,w_{k-1}\} = \{n | n < w'_j=w_i \text{ but } n \notin w_1, \dots ,w_{k-1}\},$
   so $c_j = c'_j$.
\end{proof}

\subsection{Lehmer Tableaux}

In general, since the first $m$ values of a permutation's Lehmer code depend only on the first $m$ values of the permutation, we can also define the Lehmer code of a partial permutation $w= w_1 w_2 \dots w_m$ as 
$code(w) = c_1 c_2 \dots c_m$ where
$$c_i = \#\{n | n < w_i \text{ but } n \notin \{w_1, \dots ,w_{i-1}\}\}.$$
Then taking the Lehmer code gives us a bijection from partial permutations (of $S_{\infty}$) of length $m$ to elements of $\mathbb{Z}^m_{\geq 0}$. From an inversions tableau $T$, we obtain a Lehmer tableaux by replacing the partial permutation in each column with its corresponding Lehmer code.

\begin{definition}
Let $T$ be an inversions Tableaux for a permutation $w$. The \emph{Lehmer tableaux} of $T$ is the tableau $\Lambda(T)$ defined as follows. For each $(i,j)\in\Inv(w)$, let $\Lambda(T)(i,j)$ be the number of integers $k\in 1,2, \dots, T(i,j)$ that do not appear below box $(i,j)$ in column $j$ of $T$.  
Let $\mathcal{LT}(w)=\{\Lambda(T) | T \in \mathcal{IT}(w)\}$ denote the set of Lehmer tableaux for $w$.
\end{definition}

Then for an inversions tableau $T$, and column $c$, we have that $code(word_c(T)) = word_c(\Lambda(T)).$ 

\begin{definition}
    Let $L$ be the Lehmer tableau of some inversions tableau $T$. For $C_{ij}$ a chute move on $T$, we define the corresponding \emph{chute move} on $L$ to be
    $C_{ij} L = \Lambda (C_{ij} T).$
\end{definition}

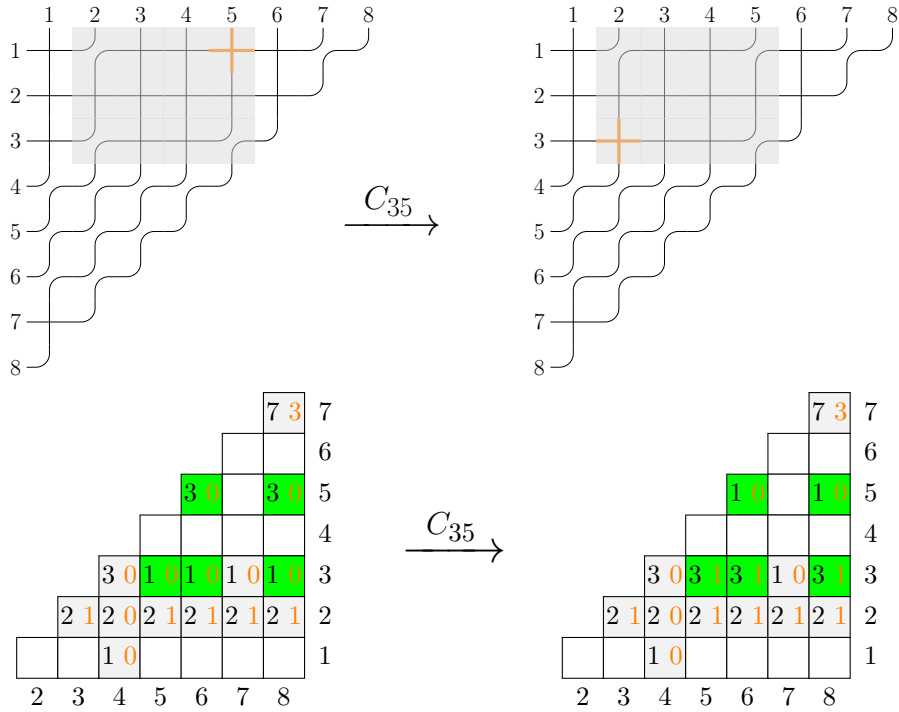
\begin{figure}[h]
\centering
$
\begin{matrix}
\scalebox{0.6}{
\begin{tikzpicture}
\pdcaps{1/8,2/7,3/6,4/5,5/4, 6/3,7/2,8/1}
\crosses{1/1,1/3,1/4,1/6,
         2/1,2/2,2/3,2/4,2/5,2/6,
         3/1,3/3,3/4,
         7/1}
\crosses[orange,line width=2pt]{1/5}
\bumps{1/2,1/7, 3/2, 3/5, 4/1,4/2,4/3,4/4,5/1,5/2,5/3, 6/1,6/2}   
\shadedsquares[gray!30]{1/2,1/3,1/4,1/5,2/2,2/3,2/4,2/5,3/2,3/3,3/4,3/5}
\node at (8,-4) {\Huge $\xrightarrow{\ C_{35} \ }$};
\foreach \i in {1,...,8} {
	\node[left] at (0,-\i+0.5) {\Large{\i}};
    \node[above] at (\i-0.5,0) {\Large \i};
}
\end{tikzpicture}
}
&
\ \ 
&
\scalebox{0.6}{
\begin{tikzpicture}
\pdcaps{1/8,2/7,3/6,4/5,5/4, 6/3,7/2,8/1}
\crosses{1/1,1/3,1/4,1/6,
         2/1,2/2,2/3,2/4,2/5,2/6,
         3/1,3/3,3/4,
         7/1}
\crosses[orange,line width=2pt]{3/2}
\bumps{1/2,1/5,1/7, 3/5, 4/1,4/2,4/3,4/4,5/1,5/2,5/3, 6/1,6/2}   
\shadedsquares[gray!30]{1/2,1/3,1/4,1/5,2/2,2/3,2/4,2/5,3/2,3/3,3/4,3/5}
\foreach \i in {1,...,8} {
	\node[left] at (0,-\i+0.5) {\Large \i};
    \node[above] at (\i-0.5,0) {\Large \i};
}
\end{tikzpicture}
}
\end{matrix}
$

$
\begin{matrix}
\begin{ytableau}
\none & \none & \none & \none & \none & \none & *(gray!10) 7 \ \textcolor{orange}{3} & \none[7] \\
\none & \none & \none & \none & \none & *(white) & *(white) & \none[6] \\
   \none & \none & \none & \none & *(green) 3 \ \textcolor{orange}{0} & *(white) & *(green) 3 \ \textcolor{orange}{0} & \none[5] \\
  \none & \none & \none & *(white) & *(white) &  *(white) & *(white) & \none[4] \\
  \none & \none & *(gray!10) 3 \ \textcolor{orange}{0} & *(green) 1 \ \textcolor{orange}{0} & *(green) 1 \ \textcolor{orange}{0} &  *(gray!10) 1 \ \textcolor{orange}{0} & *(green) 1 \ \textcolor{orange}{0} & \none[3] \\
\none & *(gray!10) 2 \ \textcolor{orange}{1} & *(gray!10) 2 \ \textcolor{orange}{0} & *(gray!10) 2 \ \textcolor{orange}{1} & *(gray!10) 2 \ \textcolor{orange}{1} &  *(gray!10)2 \ \textcolor{orange}{1} & *(gray!10) 2 \ \textcolor{orange}{1} & \none[2] \\
  *(white) & *(white) & *(gray!10) 1 \ \textcolor{orange}{0}& *(white) & *(white) &  *(white) & *(white) & \none[1] \\
  \none[2] & \none[3] & \none[4] & \none[5] & \none[6] &\none[7] & \none[8]\\
\end{ytableau}
&
\mathlarger{\mathlarger{\mathlarger{\mathlarger{ \ \ \xrightarrow{\ C_{35} \ } \ \ }}}}
&
\begin{ytableau}
\none & \none & \none & \none & \none & \none & *(gray!10) 7 \ \textcolor{orange}{3} & \none[7] \\
\none & \none & \none & \none & \none & *(white) & *(white) & \none[6] \\
   \none & \none & \none & \none & *(green) 1 \ \textcolor{orange}{0} & *(white) & *(green) 1 \ \textcolor{orange}{0} & \none[5] \\
  \none & \none & \none & *(white) & *(white) &  *(white) & *(white) & \none[4] \\
  \none & \none & *(gray!10) 3 \ \textcolor{orange}{0} & *(green) 3 \ \textcolor{orange}{1} & *(green) 3 \ \textcolor{orange}{1} &  *(gray!10) 1 \ \textcolor{orange}{0} & *(green) 3 \ \textcolor{orange}{1} & \none[3] \\
\none & *(gray!10) 2 \ \textcolor{orange}{1} & *(gray!10) 2 \ \textcolor{orange}{0} & *(gray!10) 2 \ \textcolor{orange}{1} & *(gray!10) 2 \ \textcolor{orange}{1} &  *(gray!10)2 \ \textcolor{orange}{1} & *(gray!10) 2 \ \textcolor{orange}{1} & \none[2] \\
  *(white) & *(white) & *(gray!10) 1 \ \textcolor{orange}{0}& *(white) & *(white) &  *(white) & *(white) & \none[1] \\
  \none[2] & \none[3] & \none[4] & \none[5] & \none[6] &\none[7] & \none[8]\\
\end{ytableau}
\end{matrix}
$

\caption{Top: The chute move $C_{35}$ on pipe dreams. Bottom: The same chute move on the corresponding inversions tableaux (consisting of the numbers on the left of each box in black) and Lehmer tableaux (consisting of the numbers on the right of each box in orange).}
\end{figure}

\begin{lemma}
   Let $L$ be a Lehmer tableau $L$, and $C_{ij}$ be chute move on $L$ (and therefore on its corresponding inversions tableau). Then for each column $c$,
   either $word_c(C_{ij}L) = word_c(L)$ or $word_c(C_{ij}L) = word_c(L) + e_i$.
 
\end{lemma}
\begin{proof}
    This follows directly from Lemma \ref{lemma:extended column} and Theorem \ref{Thrm: mediocre Bruhat}.
\end{proof}

\begin{corollary}Let $T$ and $T'$ be inversions tableaux for a permutation $w$ such that $T'$ is obtained from $T$ by a series of chute moves. Then for all $(i,j) \in \Inv(w)$, we have $\Lambda(T)(i,j) \leq \Lambda(T')(i,j)$.\end{corollary}

In upcoming work with Defant, Mularczyk, Nguyen, and Tung, we show that in fact $T'$ is obtained from $T$ by a series of chute moves \emph{if and only if} $\Lambda(T)(i,j) \leq \Lambda(T')(i,j)$ for all $(i,j) \in w$. This provides a global description of the partial order on a chute move poset that, when combined with the local description coming from chute moves, allows us to prove that chute move posets are lattices.

Let $w \in S_n$ and $L$ be a filling of the shaded boxes in the inversions diagram of $w$ with numbers in $\mathbb{Z}_{\geq 0}$. We would like to understand better what conditions to impose on the filling to make $L$ a Lehmer tableau. In otherwords, let $T$ be the tableaux given by $word_c(T) = code^{-1}(word_c(L))$ for all columns $c$. We want to know what conditions on $L$ will guarantee that $T$ is an inversions tableau. For the rest of this section, we will assume that $L$ and $T$ are as above.

Since $code^{-1}(word_c(L))$ is always a partial permutation, we will always have that $T$ is column strict.

\begin{lemma}
    Let $i_1 < i_2 < \dots < i_m$ be all the rows in $L$ for which the box $(i_j,c)$ is shaded in column $c > i_m$. Then $T(i_j,c) \leq i_j$ for all $1 \leq j \leq m$ if and only if $L(i_j,c) \leq i_j - j$ for all $j$.
\end{lemma}

\begin{proof}
    First, suppose $L(i_j,c) \leq i_j - j$. Then 
    $$T(i_j,c) \leq L(i_j,c) + 1 + \#\{\text{shaded boxes below } i_j \text{ in column } c\} = L(i_j,c) + j.$$

 Next suppose for all $j'<j$ we have $L(i_j',c) \leq i_j' - j'$, but $L(i_j,c) \geq i_j - j$. Assume among the set of shaded boxes $S = (i_1,c), \dots, (i_j,c)$ that $k$ of them have entries less than $T(i_j,c)$ for some $0 \leq k < j$. Then $T(i_j,c) = k + L(i,j) + 1$. There are $j-k$ boxes in $S$ with entries at least $T(i,j)$. Thus the largest entry contained by one of these boxes is at least $k + L(i,j) + (j-k) + 1 = L(i_j,c) + j > i_j$. This is larger than the row index of the box it is in.  
\end{proof}

\begin{question}
    What conditions on $L$ will guarantee that $T$ is balanced?
\end{question}

\section{Acknowledgements} 
Thank you to Colin Defant, Hanna Mularczyk, Foster Tom and Katherine Tung for help with editing. Thanks to Maya Sankar for writing the TikZ commands to make nicer pipe dreams. Thank you to Dora Woodruff and Foster Tom for ideas about the lexicographically minimal inversions tableaux, and thanks to Yuan Yao for helpful comments inspiring mediocre Bruhat order.

\printbibliography 

\end{document}